\documentclass[11pt]{article}

\usepackage{paralist}
\usepackage[margin=3.0cm]{geometry} 
\usepackage{graphicx} 

\usepackage{amssymb, amsfonts, amsmath, amscd, amsthm} 
\usepackage{mathtools} 
\usepackage{titlesec} 
\usepackage{csquotes} 
\usepackage[shortlabels]{enumitem} 
\usepackage{hyperref} 
\usepackage{mathrsfs} 
\usepackage{bbm}
\usepackage{multicol}
\usepackage[misc]{ifsym}
\usepackage{comment}
\usepackage{soul}

\hypersetup{
     colorlinks   = true,
     citecolor    = green
}

\numberwithin{equation}{section}

\titleformat*{\section}{\Large \scshape\center}
\titleformat*{\subsection}{\fontsize{14}{14} \sffamily}

\theoremstyle{plain}
\newtheorem{theorem}{Theorem}[section]
\newtheorem*{theorem*}{Theorem}

\newtheorem{lemma}[theorem]{Lemma}
\newtheorem{proposition}[theorem]{Proposition}
\newtheorem{thm*}{Theorem}
\newtheorem{prop*}{Proposition}
\newtheorem{corollary}[theorem]{Corollary}

\theoremstyle{definition}
\newtheorem{definition}[theorem]{Definition}
\newtheorem*{definition*}{Definition}
\newtheorem{example}[theorem]{Example}

\newtheorem{assumption}{Assumption}

\theoremstyle{remark}
\newtheorem{remark}[theorem]{Remark}
\newtheorem{remarks}[theorem]{Remarks}

\newcommand{\BUC}{\operatorname{BUC}}

\allowdisplaybreaks

\newcommand{\C}{\mathbb{C}}

\newcommand{\N}{\mathbb{N}}
\newcommand{\R}{\mathbb{R}}
\newcommand{\T}{\mathbb{T}}
\newcommand{\Z}{\mathbb{Z}}
\newcommand{\1}{\mathbbm{1}}

\newcommand{\Fc}{\mathcal{F}}
\newcommand{\Hc}{\mathcal{H}}
\newcommand{\Jc}{\mathcal{J}}

\newcommand{\Lc}{\mathcal{L}}

\DeclareMathOperator{\BO}{BO}
\DeclareMathOperator{\BDO}{BDO}

\DeclareMathOperator{\supp}{supp}
\DeclareMathOperator{\Co}{Co}

\newcommand{\from}{\colon}

\newcommand{\scpr}[2]{\left\langle #1, #2 \right\rangle}
\renewcommand{\sp}{\scpr}
\newcommand{\abs}[1]{\left\lvert#1\right\rvert}

\newcommand{\set}[1]{\left\{ #1\right\}}

\newcommand{\vertiii}[1]{{\left\vert\kern-0.25ex\left\vert\kern-0.25ex\left\vert #1
    \right\vert\kern-0.25ex\right\vert\kern-0.25ex\right\vert}}

\renewcommand{\epsilon}{\varepsilon}

\begin{document}
\pagenumbering{gobble}
\title{Band-dominated and Fourier-band-dominated operators on locally compact abelian groups}
\author{Robert Fulsche and Raffael Hagger}
\date{\today}
\maketitle

\begin{abstract}
    By relating notions from quantum harmonic analysis and band-dominated operator theory, we prove that over any locally compact abelian group $G$, the operator algebra $\mathcal C_1$ from quantum harmonic analysis agrees with the intersection of band-dominated operators and Fourier band-dominated operators. As an application, we characterize the compactness of operators acting on $L^2(G)$ and compare it with previous results in the discrete case. In particular, our results can be seen as a generalization of the limit operator concept to the non-discrete world. Moreover, we briefly discuss property $A'$ for arbitrary locally compact abelian groups.
\end{abstract}



\pagenumbering{arabic}

\section{Introduction}

On $\ell^2(\N)$ every bounded linear operator can be represented as an infinite matrix. If this matrix only contains finitely many non-zero diagonals, it is called a band matrix and the corresponding operator is called a band operator. These band operators form an algebra and thus taking the closure, one obtains a Banach algebra (even a $C^*$-algebra, but this is usually not so relevant in this theory) called the band-dominated operators \cite{Lindner2006,Rabinovich_Roch_Silbermann2004}. The algebra of band-dominated operators contains a lot of operators of interest such as Jacobi operators (that is, discrete Schr\"odinger operators), discretizations of differential operators, Toeplitz operators with continuous symbols as well as all compact operators. Band-dominated operators are therefore a natural class of operators to study from an algebraic point of view. It turns out that this class is exactly the right algebra for studying compactness and Fredholm properties. For technical reasons let us embed $\ell^2(\N)$ into $\ell^2(\Z)$ here and consider operators on $\ell^2(\Z)$ instead. The idea is now quite simple. Whether a band operator is Fredholm does not depend on any finite number of entries of the matrix. It follows that all the information must somehow be stored at infinity. To get there, we can shift the matrix along the diagonal and take the limit in the strong operator topology. Now obviously this limit does not have to exist, but a simple Bolzano-Weierstrass argument shows that there is at least a subsequence along which the strong limit exists. These limits are called limit operators. It is then quite easy to see that a band-dominated operator is compact if and only if all of its limit operators are $0$. Far more difficult to show but related is the fact that a band-dominated operator is Fredholm if and only if all limit operators are invertible \cite{Lindner_Seidel}.

The latter result by Lindner and Seidel, which solved a problem that was open for decades, sparked a lot of interest in this area. Several generalizations of this result have since been obtained. Maybe most notably, shortly after the initial result, which was obtained for $\ell^p(\Z^N)$, \v{S}pakula and Willett \cite{Spakula_Willett} generalized the above results to $\ell^p(X)$, where $X$ is a strongly discrete metric space of bounded geometry that satisfies Yu's property A \cite{Yu}. This also revealed some interesting connections to the coarse Baum--Connes conjecture, but we want to focus on one particular aspect here. As a side result, \v{S}pakula and Willett showed that if the space does not satisfy property A, then there must be a non-compact band-dominated operator such that all limit operators vanish. Such operators (compact or not) are called ghost operators, hence another way of putting it would be: The strongly discrete metric space of bounded geometry satisfies Yu's property A if and only if there are no non-compact ghosts.

So far, everything took place on a discrete metric space despite the fact that it is straightforward to generalize the notion of band-dominated operators to arbitrary metric spaces. In fact, just think of infinite matrices as the kernel of a discrete integral operator. Then the generalization to arbitrary integral operators is fairly obvious, and in fact, restricting to integral operators is not necessary at all and not even a metric is really needed. We refer to Section \ref{sec:operator_algebras} for the precise definition. But even though there are some tricks (see \cite{Lindner2006}, for instance) to study certain classes of operators over non-discrete spaces, such as identifying $L^p(\R) \cong \ell^p(\Z,L^p([0,1)))$, an equally satisfactory characterization of compactness has not yet been obtained in this setting. An attempt was made in \cite{HaggerSeifert}, where a replacement for property A was proposed, called property A', which is more versatile for our operator theoretic problems. Moreover, it is shown there that under very mild assumptions, the two conditions are actually equivalent. However, one major problem remained. There are plenty of ghost operators in the non-discrete setting. Just think of a multiplication operator on $L^2(\R)$ with a symbol that vanishes at infinity. Then even without going into the technical details, it is fairly obvious that all limit operators must vanish. However, such a multiplication operator is compact if and only if the symbol is $0$.

So, if it does not even work for the simplest operators, maybe all hope is lost? Not quite. The idea of \cite{HaggerSeifert} was to add a projection onto a subspace into the equation, which at least deals with cases where the operators of interest are not actually defined on the $L^p$-space itself but rather on a closed subspace such as Fock or Bergman spaces (besides \cite{HaggerSeifert}, see also \cite{Fulsche_Hagger} and \cite{Hagger2019} for how this works out). However, this is still not very satisfactory, as the very basic case of operators on $L^2(\R)$ can still not be tackled this way despite the fact that $\R$ obviously satisfies all the reasonable geometric assumptions. In this paper, we now propose a new strategy using \emph{quantum harmonic analysis} (QHA). The tools of QHA has proven to be rather successful for many different problems in recent times; see~\cite{Fulsche2020, Fulsche_Galke2025, Luef_Skrettingland2021, werner84} and the references therein. In QHA a certain group action $\alpha$ of an \emph{abelian phase space}, acting on the bounded linear operators of a Hilbert space, is introduced. This $\alpha$ can be seen as the shifting process introduced above. Usually, the abelian phase space is chosen of the form $\Xi = G \times \widehat{G}$ and the Hilbert space is $\mathcal H = L^2(G)$. Here, $G$ is a locally compact abelian (lca) group and $\widehat{G}$ denotes its Pontryagin dual. A particular role in QHA is then played by the C$^\ast$-algebra
\begin{align*}
    \mathcal C_1(\mathcal H) := \{ A \in \mathcal L(\mathcal H): ~\| A - \alpha_z(A)\|_{op} \to 0, ~\text{ as } z \to e\},
\end{align*}
where $e$ is the neutral element of the group $\Xi$. For $G := \mathbb R^{d}$ (and hence $\Xi = \mathbb R^{2d}$), $\mathcal C_1(\mathcal H)$ has spectral properties similar to band-dominated operators, cf.~\cite{Fulsche2024,Fulsche_Hagger}. In fact, it is well known that in case $G = \Z^N$ (hence $\Xi = \Z^N \times \T^N$) the algebra of band-dominated operators actually agrees with $\mathcal C_1(\ell^2(\mathbb Z))$, cf.~\cite[Theorem 2.1.6]{Rabinovich_Roch_Silbermann2004}.

In this paper, we explore the precise connection between band-dominated operators and the operator algebra $\mathcal C_1(L^2(G))$. As a consequence, this will provide a compactness characterization in the same spirit as the original result on $\ell^p(\Z)$. Our two main results are the following:

\begin{thm*}
    Let $G$ be a locally compact abelian group. Then the following equality holds:
    \[\mathcal C_1(L^2(G)) = \operatorname{BDO}(L^2(G)) \cap \mathcal F^{-1} \operatorname{BDO}(L^2(\widehat{G}))\mathcal F.\]
\end{thm*}

Here, $\Fc \from L^2(G) \to L^2(\widehat{G})$ denotes the Fourier transform, and in the following $\partial\Xi$ will denote the boundary of $\Xi$ in a certain compactification of $\Xi$.

\begin{thm*}
    Let $G$ be a locally compact abelian group, $\Xi := G \times \widehat{G}$ and $B \in \Lc(L^2(G))$. Then $B$ is compact if and only if $B \in \BDO(L^2(G))$, $\Fc B \Fc^{-1} \in \BDO(L^2(\widehat{G}))$ and $\alpha_z(B) = 0$ for every $z \in \partial\Xi$. 
\end{thm*}

Indeed, what we are going to prove is a bit more detailed than the above statements, and the precise formulation needs some more technicalities that we only mention in later parts of the paper. To briefly summarize this, we prove the result not only for the Hilbert space $L^2(G)$, but also for operators on the modulation spaces $M^p(G)$. If the group $G$ is discrete, then $M^p(G) = \ell^p(G)$, putting the results in alignment with the classical results for band-dominated operators over sequence spaces, noting that $\BDO(M^p(\widehat{G})) = \Lc(M^p(\widehat{G}))$ and $\partial\Xi = \partial G \times \widehat{G}$ in this case. When the group is not discrete, then $M^p(G) \neq L^p(G)$ for $p \neq 2$. Further, we are actually able to classify the band-dominated operators (and not only their intersection with Fourier-BDO) in terms of continuity with respect to a certain group action, leading to the above theorem. Let us also mention that we do not assume that $G$ is metrizable. In fact, we do not even need first countable. This contrasts to previous work such as \cite{HaggerSeifert} and \cite{Spakula_Willett}, where the metric is an essential part of the theory. Instead of a metric, we will use the uniform structure that the lca group provides. We note that this involves some technical issues when the group is not second countable. Nevertheless, all these problems can be overcome, mainly by making an appropriate change in the definition of what the space $L^\infty$ is. We shall not elaborate on the details of this here and instead refer to the discussions in \cite[End of Section 3]{Fulsche_Galke2025} or \cite[Section 2.3]{Folland2016}.

The paper is organized as follows. In Section \ref{sec:coorbit}, we describe the basic setting of the present work, combined with an introduction to coorbit spaces, as suitable for the precise setting we are working with. Since we also want to make this section available as a reference for later work, we put some more details there than would be strictly necessary for this paper. Section \ref{sec:operator_algebras} contains the precise definitions and statements of our main results as well as their proofs. We also provide some simple examples to indicate how our results can be applied. Finally, we add a short appendix to show that lca groups always satisfy property A'. This explains why unlike in the case of \v{S}pakula and Willett \cite{Spakula_Willett} no additional geometrical assumption is necessary.

\section{Coorbit spaces}\label{sec:coorbit}

Let us first make some basic assumptions that we will adhere to for the rest of the paper.

\begin{assumption}
    We assume that $\Xi$ is a locally compact abelian (lca) group with identity $e$ and $m \from \Xi \times \Xi \to S^1$ is a separately continuous 2-cocycle on $\Xi$ which is a Heisenberg multiplier, that is, it satisfies
    \[m(xy,z)m(x,y) = m(x,yz)m(y,z)\]
    for $x,y,z \in \Xi$ and $x \mapsto \sigma(x,\cdot) := m(x,\cdot)/m(\cdot,x)$ is a topological group isomorphism from $\Xi$ to $\widehat{\Xi}$. 

    In addition, we assume that $m(x,y) = m(x^{-1}, y^{-1})$ for each $x, y\in \Xi$ as well as $m(e,e) = 1$.
\end{assumption}
These assumptions ensure that there is a unique (up to unitary equivalence) irreducible projective unitary representation $(\mathcal H, U_x)$ with $m$ as a cocycle, that is,
    \[U_xU_y = m(x,y)U_{xy}\]
    for $x,y \in \Xi$, cf.\ \cite[Theorem 3.3]{Baggett_Kleppner1973} for the relevant uniqueness theorem and \cite[Section 2]{Fulsche_Galke2025} for more explanations on this. Since we assumed $m(e,e) = 1$, it is not difficult to verify that $U_x^\ast = \overline{m(x,x^{-1})}U_{x^{-1}}$. Moreover, the cocycle condition and $m(e,e) = 1 $ imply $m(e,x) = 1$ for all $x \in \Xi$.
    
\begin{assumption} 
    We always add the standing assumption that the representation $(U_x)_{x \in \Xi}$ is \emph{integrable}, that is, that the set $\mathcal H_1 := \{ f \in \mathcal H: ~x \mapsto \langle f, U_x f\rangle_{\mathcal H} \in L^1(\Xi)\}$ of $\mathcal H$ is non-trivial. 
\end{assumption}

\begin{remarks}\label{rem:one}
\begin{itemize}
    \item[(a)] Note that any 2-cocycle is similar to one with $m(e,e) = 1$ (consider, e.g. $m'(x,y) = \frac{a(x)a(y)}{a(xy)}m(x,y)$, where $a(x) = \overline{m(x,e)}$). Hence, it is not a significant restriction that we assumed $m(e,e) = 1$ for convenience. On the operator theory side, the representations are related by $U_x' = a(x)U_x$, so here the assumption $m(e,e) = 1$ is not a significant restriction either.
    \item[(b)] The assumption that $m(x,y) = m(x^{-1},y^{-1})$ for all $x,y \in \Xi$ ensures that there exists a unitary operator $R \in \mathcal U(\mathcal H)$ satisfying $RU_x = U_{x^{-1}}R$. $R$ can be chosen to be self-adjoint (cf.\ \cite[Section 2]{Fulsche_Galke2025}).
    \item[(c)] The uniqueness of the irreducible representation, together with the separate continuity of the cocycle, guarantee that the representation is continuous in strong operator topology.
    \item[(d)] Note that the set $\mathcal H_1$ defined in the assumption is automatically dense in $\mathcal H$ and even contains a dense subspace of $\mathcal H$. This can be proven as for standard unitary representations, cf.\ \cite[Proposition 14.5.1]{Dixmier_1977}. Further, it is itself indeed a subspace, as we will see later.
    \item[(e)] Although the groups under consideration are abelian, we decided to write the group operation in a multiplicative way (that is, $xy$ instead of $x+y$ and $x^{-1}$ instead of $-x$). The reason for this is that we are particularly interested in the special case $\Xi = G \times \widehat{G}$ and $\widehat{G} = \set{\chi \from G \to S^1 : \chi \text{ is a continuous group homomorphism}}$ is necessarily multiplicative (see Section \ref{sec:operator_algebras}). The neutral element of $\widehat{G}$ will always be denoted by $1$.
    \item[(f)] If $\Xi$ is of the form $G \times \widehat{G}$, then
    \[m((x,\varphi),(y,\psi)) := \overline{\psi(x)}\]
    and a corresponding projective representation is given by
    \[U_{(x,\varphi)} \from L^2(G) \to L^2(G), \quad \big[U_{(x,\varphi)}f\big](y) := \varphi(y)f(x^{-1}y)\]
    for $(x,\varphi) \in \Xi = G \times \widehat{G}$.
\end{itemize}
\end{remarks}

We are now going to describe the basics of coorbit theory in the setting of our projective representations. The theory of coorbit spaces has been developed in the papers \cite{Feichtinger_Groechenig1988, Feichtinger_Groechenig1989a, Feichtinger_Groechenig1989b} by H.~G.~Feichtinger and K.~Gr\"{o}chenig. A detailed discussion for the case of projective representations has been given in \cite{Christensen1996}.

In the following, we are going to describe the essentials of the theory of coorbit theory that are necessary for the following sections of the paper. Since we do not assume the reader to be familiar with this theory, we decided to formulate the results in precisely the setting we are working in (in contrast to the more general works cited above) and to give full proofs of all these results. 
We do not want to hide that our discussion of coorbit theory is based on the very nice survey \cite{Berge2022}, and most of the results between Lemma \ref{lemma:WaveletTrafoIsometric} and Proposition \ref{prop:rangeWaveletTrafo} are based on the analogous results in that survey without giving explicit reference at each occasion.

Fix $0 \neq \varphi_0 \in \mathcal H_1$. For $f \in \mathcal H$ and $x \in \Xi$ we will write $\big[\mathcal W_{\varphi_0}(f)\big](x) := \langle f, U_x \varphi_0\rangle_{\mathcal H}$ for the \emph{Wavelet transform} with window $\varphi_0$. The Wavelet transform is injective and sends elements from $\mathcal H$ to $L^2(\Xi)$:

\begin{lemma}\label{lemma:WaveletTrafoIsometric}
    There exists a constant $c > 0$ such that for each $f, g \in \mathcal H$ and $\varphi_0, \varphi_1 \in \mathcal H_1$:
    \begin{align*} 
    \int_\Xi \big[\mathcal W_{\varphi_0}(f)\big](x) \overline{\big[\mathcal W_{\varphi_1}(g)\big](x)}\, \mathrm{d}x = c \langle f, g\rangle_{\mathcal H} \langle \varphi_1, \varphi_0\rangle_{\mathcal H}.
 \end{align*}
\end{lemma}
Note that by suitably normalizing the Haar measure $\lambda$ on $\Xi$, we can enforce $c = 1$ in the above lemma. We will always choose the normalization of the Haar measure so that this is satisfied and just write $\mathrm{d}x$ instead of $\mathrm{d}\lambda(x)$.

The previous lemma is just a reformulation of Godement's orthogonality relations, cf.\ \cite[Theorem 2.4]{Fulsche_Galke2025}. We set $\| f\|_{2, \varphi_0} := \| \mathcal W_{\varphi_0}(f)\|_{L^2(\Xi)}$. Furthermore, by the continuity properties of the representation, each element of $\mathcal W_{\varphi_0}(\mathcal H)$ is a bounded and continuous function, which means that we can evaluate the functions at specific points.

\begin{lemma} \label{lem:V_x}
    Let $f \in \mathcal H$ and $y \in \Xi$. Then, $\mathcal W_{\varphi_0}(U_x f)(y) = \frac{m(x,x^{-1})}{m(x^{-1},y)} \mathcal W_{\varphi_0}(f)(x^{-1}y)$.
\end{lemma}

\begin{proof}
The statement follows from
\[\big[\mathcal W_{\varphi_0}(U_x f)\big](y) = \langle U_x f, U_y \varphi_0\rangle_{\mathcal H} = \frac{m(x,x^{-1})}{m(x^{-1},y)} \langle f, U_{x^{-1}y}\varphi_0\rangle_{\mathcal H}.\qedhere\]
\end{proof}

Motivated by this, we will set $V_x f(y) := \frac{m(x,x^{-1})}{m(x^{-1},y)}f(x^{-1}y)$ for functions $f: \Xi \to \mathbb C$. By this convention, the Wavelet transform intertwines $V_x$ and $U_x$. Clearly, $x \mapsto V_x$ is again a projective unitary respresentation on $L^2(\Xi)$ which is continuous in the strong operator topology and $V_xV_y = m(x,y)V_{xy}$ for all $x,y \in \Xi$. Moreover, $V_x$ is a surjective isometry on every $L^p(\Xi)$, $p \in [1,\infty]$, with
\[V_x^{-1} = \overline{m(x,x^{-1})}V_{x^{-1}}.\]

In the following, the twisted convolution of two functions $f, g: \Xi \to \mathbb C$ will be a useful operation. It is defined as
\begin{align*}
    (f \ast' g)(x) = \int_\Xi f(y) g(y^{-1}x)m(y,y^{-1}x) \, \mathrm{d}y.
\end{align*}
This twisted convolution is associative and satisfies the same Young-type inequalities as the standard convolution, simply because $|(f \ast' g)(x)| \leq (|f| \ast |g|)(x)$. The following identity will play an important role:
\begin{lemma} \label{lem:W_twisted_convolution}
    Let $\varphi_0, \varphi_1, \varphi_2, \varphi_3 \in \mathcal H$. Then, 
    \begin{align*}
        \mathcal W_{\varphi_0}(\varphi_3) \ast' \mathcal W_{\varphi_1}(\varphi_2) = \langle \varphi_2, \varphi_0\rangle_{\mathcal H} \mathcal W_{\varphi_1}(\varphi_3).
    \end{align*}
\end{lemma}

\begin{proof}
    We have
    \begin{align*}
        \big[\mathcal W_{\varphi_0}(\varphi_3) \ast' \mathcal W_{\varphi_1}(\varphi_2)\big](x) &= \int_\Xi \langle \varphi_3, U_y \varphi_0\rangle_{\mathcal H} \langle \varphi_2, U_{y^{-1}x}\varphi_1\rangle_{\mathcal H} m(y,y^{-1}x) \, \mathrm{d}y
    \end{align*}
    and
    \begin{align*}
        \langle \varphi_2, U_{y^{-1}x} \varphi_1\rangle_{\mathcal H} &= m(y^{-1},x) \overline{m(y,y^{-1})}\langle U_y \varphi_2, U_x \varphi_1\rangle_{\mathcal H}\\
        &= m(y^{-1},x)\overline{m(y,y^{-1})} \overline{\langle U_x \varphi_1, U_y \varphi_2\rangle_{\mathcal H}},
    \end{align*}
    which yields
    \begin{align*}
        \big[\mathcal W_{\varphi_0}(\varphi_3) \ast' \mathcal W_{\varphi_1}(\varphi_2)\big](x) &= \int_\Xi \langle \varphi_3, U_y \varphi_0\rangle_{\mathcal H} \overline{\langle U_x \varphi_1, U_y \varphi_2\rangle_{\mathcal H}} \frac{m(y^{-1},x)m(y,y^{-1}x)}{m(y,y^{-1})}~\mathrm{d}y.
    \end{align*}
    The cocycle property of $m$ implies
    \[m(y,y^{-1}x)m(y^{-1},x) = m(e,x)m(y,y^{-1}) = m(y,y^{-1}),\]
    hence, using Lemma \ref{lemma:WaveletTrafoIsometric},
    \begin{align*}
        \big[\mathcal W_{\varphi_0}(\varphi_3) \ast' \mathcal W_{\varphi_1}(\varphi_2)\big](x) &= \int_\Xi \langle \varphi_3, U_y \varphi_0\rangle_{\mathcal H} \overline{\langle U_x \varphi_1, U_y \varphi_2\rangle_{\mathcal H}}~\mathrm{d}y
        = \langle \varphi_2, \varphi_0\rangle_{\mathcal H} \overline{\langle U_x \varphi_1, \varphi_3\rangle_{\mathcal H}}\\
        &= \langle \varphi_2, \varphi_0\rangle_{\mathcal H} \mathcal W_{\varphi_1}(\varphi_3)(x).
    \end{align*}
    This finishes the proof.
\end{proof}

\begin{lemma} \label{lem:orthogonal_projection}
    Let $0 \neq \varphi_0 \in \mathcal H_1$ and $f \in L^2(\Xi)$. Then $f \in \mathcal W_{\varphi_0}(\mathcal H)$ if and only if $f = \frac{1}{\| \varphi_0\|^2_{\mathcal H}}f \ast' \mathcal W_{\varphi_0}(\varphi_0)$. In addition, the orthogonal projection from $L^2(\Xi)$ to $\mathcal W_{\varphi_0}(\mathcal H)$ is given by $f \mapsto \frac{1}{\|\varphi_0\|_{\mathcal H}^2}f \ast' \mathcal W_{\varphi_0}(\varphi_0)$.
\end{lemma}

\begin{proof}
    Let $g \in \mathcal H$. Then, by Lemma \ref{lem:W_twisted_convolution}, $\mathcal W_{\varphi_0}(g)$ clearly satisfies
    \begin{align*}
        \mathcal W_{\varphi_0}(g) = \frac{1}{\| \varphi_0\|_{\mathcal H}^2} \mathcal W_{\varphi_0}(g) \ast' \mathcal W_{\varphi_0}(\varphi_0).
    \end{align*}
    For the reverse statement, observe that by Lemma \ref{lemma:WaveletTrafoIsometric} (recalling that we normalized the Haar measure on $\Xi$ so that $c = 1$ in that lemma), the map $\mathcal W_{\varphi_0}/\| \varphi_0\|_{\mathcal H} \from \mathcal H \to L^2(\Xi)$ is an isometry and therefore has closed range. Hence, there exists an orthogonal projection from $L^2(\Xi)$ to $\mathcal W_{\varphi_0}(\mathcal H)$, which is given by $\frac{1}{\| \varphi_0\|_{\mathcal H}^2} \mathcal W_{\varphi_0}\mathcal W_{\varphi_0}^\ast$. Now, for $f \in L^2(\Xi)$ we have
    \[\big[\mathcal W_{\varphi_0} \mathcal W_{\varphi_0}^\ast(f)\big](x) =  \langle \mathcal W_{\varphi_0}^\ast f, U_x \varphi_0\rangle_{\mathcal H} = \langle f, \mathcal W_{\varphi_0} U_x \varphi_0\rangle_{L^2(\Xi)}.\]
    From Lemma \ref{lem:V_x} we know that
    \begin{align*}
        \big[\mathcal W_{\varphi_0} U_x \varphi_0\big](y) &= m(x,x^{-1}) \overline{m(x^{-1},y)}\big[\mathcal W_{\varphi_0}(\varphi_0)\big](x^{-1}y).
    \end{align*}
    Therefore,
    \begin{align*}
        \big[\mathcal W_{\varphi_0} \mathcal W_{\varphi_0}^\ast(f)\big](x) &= \int_\Xi f(y) \overline{m(x,x^{-1})}m(x^{-1},y) \overline{\big[\mathcal W_{\varphi_0}(\varphi_0)\big](x^{-1}y)} \, \mathrm{d}y
    \end{align*}
    Using the formula $\overline{\big[\mathcal W_{\varphi_0}(\varphi_0)\big](t)} = m(t,t^{-1}) \big[\mathcal W_{\varphi_0}(\varphi_0)\big](t^{-1})$, we arrive at
    \begin{align*}
       \frac{1}{\| \varphi_0\|_{\mathcal H}^2} \big[\mathcal W_{\varphi_0} \mathcal W_{\varphi_0}^\ast(f)\big](x) = \frac{1}{\| \varphi_0\|_{\mathcal H}^2}  \int_\Xi f(y) \big[\mathcal W_{\varphi_0}(\varphi_0)\big](y^{-1}x) \frac{m(x^{-1},y)m(x^{-1}y,y^{-1}x)}{m(x,x^{-1})} \, \mathrm{d}y.
    \end{align*}
    Using
    \begin{equation} \label{eq:some_cocycle_equation}
        m(x^{-1}y,y^{-1}x)m(x^{-1},y) = m(x^{-1},x)m(y,y^{-1}x) = m(x,x^{-1})m(y,y^{-1}x)
    \end{equation}
    finishes the proof.
\end{proof}

\begin{corollary}
   Consider $\mathcal W_{\varphi_0}(\mathcal H)$ as a Hilbert space, endowed with the inner product inherited from $L^2(\Xi)$. Then, it is a reproducing kernel Hilbert space with kernel
    \begin{align*}
        K_x(y) = \frac{1}{\| \varphi_0\|_{\mathcal H}^2} \big[\mathcal W_{\varphi_0}(\varphi_0)\big](x^{-1}y) \frac{m(x,x^{-1})}{m(x^{-1},y)} = \frac{1}{\| \varphi_0\|_{\mathcal H}^2}  \big[\mathcal W_{\varphi_0}(U_x \varphi_0)\big](y).
    \end{align*}
\end{corollary}

We now consider $\Co_1(U) := \{ f \in \mathcal H: ~\mathcal W_{\varphi_0}(f) \in L^1(\Xi)\}$, which turns into a normed space when endowed with the norm
\begin{align*}
    \| f\|_{1, \varphi_0} := \| \mathcal W_{\varphi_0}(f)\|_{L^1(\Xi)}.
\end{align*}
The space $\Co_1(U)$ is one of the coorbit spaces in which we are interested. The general definition of a coorbit space will follow after some basic results. A priori, this space may depend on the choice of $\varphi_0$. But this is actually not the case.

\begin{proposition}
    Let $0 \neq \varphi_0, \varphi_1 \in \mathcal H_1$ and $f \in \mathcal H$ be such that $\mathcal W_{\varphi_0}(f) \in L^1(\Xi)$. Then, $\mathcal W_{\varphi_1}(f) \in L^1(\Xi)$ with $\| \mathcal W_{\varphi_1}(f)\|_{L^1(\Xi)} \leq C \| \mathcal W_{\varphi_0}(f)\|_{L^1(\Xi)}$, where the constant $C$ does not depend on $f$. 
\end{proposition}

\begin{proof}
Let $x \in \Xi$ be such that $\langle \varphi_0, U_x \varphi_1 \rangle_{\mathcal H} \neq 0$ (such $x$ always exists, because the representation $(U_x)_{x \in \Xi}$ is irreducible, hence every non-zero vector is cyclic). Then, by Lemma \ref{lem:W_twisted_convolution} and Young's inequality,
\begin{align*}
        \|\mathcal W_{\varphi_1}(f)\|_{L^1(\Xi)} &= \frac{1}{|\langle U_x \varphi_1 , \varphi_0 \rangle_{\mathcal H}|} \| \mathcal W_{\varphi_0}(f) \ast' \mathcal W_{\varphi_1}(U_x\varphi_1)\|_{L^1(\Xi)}\\
        &\leq \frac{1}{|\langle U_x \varphi_1 , \varphi_0 \rangle_{\mathcal H}|} \| \mathcal W_{\varphi_0}(f)\|_{L^1(\Xi)} \| \mathcal W_{\varphi_1}(U_x \varphi_1)\|_{L^1(\Xi)}\\
        &= \frac{1}{|\langle U_x \varphi_1 , \varphi_0 \rangle_{\mathcal H}|} \| \mathcal W_{\varphi_0}(f)\|_{L^1(\Xi)} \| \mathcal W_{\varphi_1}(\varphi_1)\|_{L^1(\Xi)},
    \end{align*}
    where the last equality follows from Lemma \ref{lem:V_x}. This finishes the proof.
\end{proof}

As an immediate consequence, we see that the space $\Co_1(U)$ does not depend on the choice of the window function $\varphi_0$, and different window functions yield equivalent norms. Note that we clearly have $\mathcal H_1 \subseteq \Co_1(U)$ by the previous result. Indeed, both sets are equal:

\begin{proposition}
    We have $\mathcal H_1 = \Co_1(U)$. In particular, $\mathcal H_1$ is a subspace of $\mathcal H$.
\end{proposition}

\begin{proof}
    We need to prove the following: If $f \in \mathcal H$ such that there exists $\varphi_0 \in \mathcal H$ with $\mathcal W_{\varphi_0}(\varphi_0) \in L^1(\Xi)$ and $\mathcal W_{\varphi_0}(f) \in L^1(\Xi)$, then $\mathcal W_f(f) \in L^1(\Xi)$. Indeed, for every $\varphi_0 \in \mathcal H_1 \setminus \set{0}$ Lemma \ref{lem:W_twisted_convolution} implies
    \begin{align*}
        \mathcal W_f f = \frac{1}{\| \varphi_0\|_{\mathcal H}^2} \mathcal W_{\varphi_0}(f) \ast' \mathcal W_f (\varphi_0).
    \end{align*}
    Since 
    \begin{align*}
        |\big[\mathcal W_f (\varphi_0)\big](x)| &= |\langle \varphi_0, U_x f\rangle_{\mathcal H}| = |\langle U_{x^{-1}} \varphi_0, f\rangle_{\mathcal H}| = |\big[\mathcal W_{\varphi_0}(f)\big](x^{-1})|,
    \end{align*}
    we see that $\mathcal W_f(\varphi_0) \in L^1(\Xi)$ such that $\mathcal W_f(f)$ can be written as the twisted convolution of two $L^1$ functions, hence is itself contained in $L^1(\Xi)$. 
\end{proof}

Note the following simple lemmata.

\begin{lemma}\label{lem:embeddingH1}
    The embedding $(\Co_1(U), \| \cdot\|_{1, \varphi_0}) \hookrightarrow \mathcal H$ is continuous with
    \[\| f\|_{\mathcal H} \| \varphi_0\|_{\mathcal H} \leq \| f\|_{1, \varphi_0}.\] 
\end{lemma}

\begin{proof}
    Again by Godement's orthogonality relation (Lemma \ref{lemma:WaveletTrafoIsometric}),
    \begin{align*}
        \| f\|_{\mathcal H}^2 \| \varphi_0\|_{\mathcal H}^2 &= \int_\Xi |\langle f, U_x \varphi_0\rangle_{\mathcal H}|^2 \, \mathrm{d}x \leq \sup_{x \in \Xi} |\langle f, U_x \varphi_0\rangle_{\mathcal H}| \int_\Xi |\langle f, U_x \varphi_0\rangle_{\mathcal H}| \, \mathrm{d}x\\
        &\leq \| f\|_{\mathcal H} \| \varphi_0\|_{\mathcal H} \| f\|_{1, \varphi_0}.
    \end{align*}
    These estimates provide the continuity of the embedding.
\end{proof}

\begin{lemma} \label{lem:Co_1_Banach_space}
Let $0 \neq \varphi_0 \in \mathcal H_1$. 
    \begin{itemize}
        \item[(i)] $L^1(\Xi) \ast' \mathcal W_{\varphi_0}(\varphi_0)$ is a closed subspace of $L^1(\Xi)$. 
        \item[(ii)] $\mathcal W_{\varphi_0}(\Co_1(U)) = L^1(\Xi) \ast' \mathcal W_{\varphi_0}(\varphi_0)$. In particular, $\Co_1(U)$ is a Banach space.
        \item[(iii)] A function $f\in L^1(\Xi)$ is contained in $\mathcal W_{\varphi_0}(\Co_1(U))$ if and only if $f = \frac{1}{\| \varphi_0\|_{\mathcal H}^2} f \ast' \mathcal W_{\varphi_0}(\varphi_0)$.
        \item[(iv)] The map $f \mapsto \frac{1}{\| \varphi_0\|_{\mathcal H}^2} f \ast' \mathcal W_{\varphi_0}(\varphi_0)$ is a continuous projection from $L^1(\Xi)$ onto the space $\mathcal W_{\varphi_0}(\Co_1(U))$.
    \end{itemize}
\end{lemma}

\begin{proof}
\begin{itemize}
    \item[(i)] The map $f \mapsto f - \frac{1}{\| \varphi_0\|_{\mathcal H}^2} f \ast'\mathcal W_{\varphi_0}(\varphi_0)$ is a bounded linear operator on $L^1(\Xi)$. Therefore, the kernel of this operator, which is of course given by
    \begin{align*}
        \set{ f \in L^1(\Xi): ~f = \frac{1}{\| \varphi_0\|_{\mathcal H}^2} f \ast' \mathcal W_{\varphi_0}(\varphi_0)},
    \end{align*}
    is a closed subspace of $L^1(\Xi)$. Since $\frac{1}{\| \varphi_0\|_{\mathcal H}^2}\mathcal W_{\varphi_0}(\varphi_0) \ast' \mathcal W_{\varphi_0}(\varphi_0) =\mathcal W_{\varphi_0}(\varphi_0)$ by Lemma \ref{lem:W_twisted_convolution}, this space agrees with $L^1(\Xi) \ast' \mathcal W_{\varphi_0}(\varphi_0)$.
    \item[(ii)] Clearly, according to Lemma \ref{lem:W_twisted_convolution}, every $g \in \Co_1(U)$ satisfies $\mathcal W_{\varphi_0}(g) = \frac{1}{\| \varphi_0\|^2} \mathcal W_{\varphi_0}(g) \ast' \mathcal W_{\varphi_0}(\varphi_0)$. On the other hand, assume that $f \in L^1(\Xi)$ satisfies $f = \frac{1}{\| \varphi_0\|^2} f \ast' \mathcal W_{\varphi_0}(\varphi_0)$. Then we also have
    \begin{align*}
        f = \frac{1}{\| \varphi_0\|_{\mathcal H}^2} f \ast' \mathcal W_{\varphi_0}(\varphi_0) = \frac{1}{\| \varphi_0\|_{\mathcal H}^4} [f \ast' \mathcal W_{\varphi_0}(\varphi_0)] \ast' \mathcal W_{\varphi_0}(\varphi_0).
    \end{align*}
    Since $f = \frac{1}{\| \varphi_0\|_{\mathcal H}^2} f \ast' \mathcal W_{\varphi_0}(\varphi_0)$ is assumed to be in $L^1(\Xi)$, and $\mathcal W_{\varphi_0}(\varphi_0) \in L^2(\Xi)$, we see from Young's convolution inequality that we also have $f \in L^2(\Xi)$. Lemma \ref{lem:orthogonal_projection} shows that there exists $g \in \mathcal H$ such that $f = \mathcal W_{\varphi_0}(g)$ and since $f \in L^1(\Xi)$ is assumed, we must have $g \in \Co_1(U)$. Hence, we have shown the equality of the spaces. Since $\Co_1(U)$ and $\mathcal W_{\varphi_0}(\Co_1(U))$ are, by definition, isometrically isomorphic, and the latter is a closed subspace of a Banach space by (i), we see that $\Co_1(U)$ is complete.
    \item[(iii)] and (iv) are now immediate from the previous arguments.\qedhere
\end{itemize}
\end{proof}

\begin{lemma}
    The operators $(U_x)_{x \in \Xi}$ form a strongly continuous and isometric projective representation of $\Xi$ on $\Co_1(U)$.
\end{lemma}

\begin{proof}
    As observed in Lemma \ref{lem:V_x}, we have that $U_x$ acts, with respect to the Wavelet transform, as $\big[\mathcal W_{\varphi_0}(U_x f)\big](y) = \frac{m(x,x^{-1})}{m(x^{-1},y)} \big[\mathcal W_{\varphi_0}(f)\big](x^{-1}y)$. Since this operation acts isometrically (with respect to the $L^1$-norm) and strongly continuously on $C_c(\Xi)$ (continuous functions with compact support), which is a dense subspace of $L^1(\Xi)$, these properties extend to $L^1(\Xi)$, hence also to every closed, translation-invariant subspace of $L^1(\Xi)$.
\end{proof}

We now denote by $\mathcal H_\infty$ the space of antilinear continuous functionals on $\mathcal H_1$. When we consider $\mathcal H_1 = \Co_1(U)$ as a Banach space (with respect to a fixed window function $\varphi_0$), $\mathcal H_\infty$ turns itself into a Banach space. We will write $\Co_\infty(U)$ for $\mathcal H_\infty$ as a space with the dual norm with respect to a fixed window function and the corresponding norm will be denoted by $\| \cdot \|_{\infty, \varphi_0}$. By Lemma \ref{lem:embeddingH1}, we have $\mathcal H \hookrightarrow \mathcal H_\infty$ continuously via
\[g \mapsto \varphi_g, \quad \varphi_g(f) := \langle g, f\rangle_{\mathcal H} = \frac{1}{\| \varphi_0\|_{\mathcal H}^2} \langle \mathcal W_{\varphi_0}(g), \mathcal W_{\varphi_0}(f)\rangle_{L^2(\Xi)}\]
for $f \in \mathcal H_1$ and $g \in \mathcal H$, with $\| \varphi_g\|_{\infty, \varphi_0} \leq \| g\|_{\mathcal H}/\|  \varphi_0\|_{\mathcal H}$.

Given $\phi \in \mathcal H_\infty$, we define $U_x \phi$ weakly as $\big[U_x(\phi)\big](f) := \phi(U_x^\ast f)$ for $f \in \mathcal H_1$. It is now not difficult to see that the operators $U_x$ act isometrically (with respect to a fixed window function) and weakly continuously on $\Co_\infty(U)$. 

For any $\phi \in \Co_\infty(U)$, we can again consider its wavelet transform $\big[\mathcal W_{\varphi_0}(\phi)\big](x) = \phi(U_x \varphi_0)$. 

\begin{lemma}
    Let $\phi \in \Co_\infty(U)$. Then $\mathcal W_{\varphi_0}(\phi) \in C(\Xi)$ and $\| \mathcal W_{\varphi_0}(\phi)\|_\infty \leq \| \phi\|_{\infty, \varphi_0} \| \varphi_0\|_{1, \varphi_0}$. 
\end{lemma}

\begin{proof}
    Continuity of $\mathcal W_{\varphi_0}(\phi)$ is clear from the continuity of the representation on $\Co_1(U)$. Furthermore,
    \[
        \abs{\big[\mathcal W_{\varphi_0}(\phi)\big](x)} = |\phi(U_x \varphi_0)|\leq \| \phi\|_{\infty, \varphi_0} \| U_x \varphi_0\|_{1, \varphi_0}.\qedhere
    \]
\end{proof}

We note that the Wavelet transform satisfies the usual covariance properties with respect to the representation:

\begin{lemma}
    For $\phi \in \Co_\infty(U)$ we have
    \[\big[\mathcal W_{\varphi_0}(U_x \phi)\big](y) = m(x, x^{-1}) \overline{m(x^{-1}, y)} \big[\mathcal W_{\varphi_0}(\phi)\big](x^{-1}y).\]
\end{lemma}

Denote by $P_{\varphi_0}$ the continuous projection from $L^1(\Xi)$ onto $\mathcal W_{\varphi_0}(\Co_1(U))$. Since $\mathcal W_{\varphi_0}$ maps isometrically from $\Co_1(U)$ to $L^1(\Xi)$, the adjoint (more precisely, the Hilbert space adjoint of $\mathcal W_{\varphi_0}: \mathcal H \to L^2(\Xi)$) maps $\mathcal W_{\varphi_0}(\Co_1(U))$ isometrically to $\Co_1(U)$. In particular, the operator $\mathcal W_{\varphi_0}^\ast P_{\varphi_0}$ maps $L^1(\Xi)$ continuously onto $\Co_1(U)$. Hence, the adjoint operator, which simply acts as $\mathcal W_{\varphi_0}$, maps $\Co_\infty(U)$ injectively to $L^\infty(\Xi)$. Hence:

\begin{lemma}
    The Wavelet transform $\mathcal W_{\varphi_0}$ is injective on $\Co_\infty(U)$.
\end{lemma}

Since the Wavelet transform maps $\mathcal W_{\varphi_0}: \Co_1(U) \to L^1(\Xi)$, we can consider the adjoint map $\mathcal W_{\varphi_0}^\ast: L^\infty(\Xi) \to \Co_\infty(U)$. We then have:

\begin{lemma}
    The adjoint of the Wavelet transform, $\mathcal W_{\varphi_0}^\ast: L^\infty(\Xi) \to \Co_\infty(U)$, satisfies the following: 
    \begin{align*}
        \mathcal W_{\varphi_0} \mathcal W_{\varphi_0}^\ast(f) &= f \ast' \mathcal W_{\varphi_0}(\varphi_0), \quad f \in L^\infty(\Xi),\\
        \frac{1}{\| \varphi_0\|_{\mathcal H}^2} \mathcal W_{\varphi_0}^\ast \mathcal W_{\varphi_0}(\phi) &= \phi, \quad \phi \in \Co_\infty(U).
    \end{align*}
\end{lemma}

\begin{proof}
    For $f \in L^\infty(\Xi)$, we have:
    \begin{align*}
        \big[\mathcal W_{\varphi_0} \mathcal W_{\varphi_0}^\ast(f)\big](x) &= \big[\mathcal W_{\varphi_0}^*(f)\big](U_x \varphi_0)\\
        &= \int_\Xi f(y) \overline{\big[\mathcal W_{\varphi_0}(U_x \varphi_0)\big](y)} \, \mathrm{d}y\\
        &= \int_\Xi f(y) \overline{\big[\mathcal W_{\varphi_0} \varphi_0\big](x^{-1}y)} \frac{m(x^{-1},y)}{m(x,x^{-1})} \, \mathrm{d}y\\
        &= \int_\Xi f(y) \big[\mathcal W_{\varphi_0}\varphi_0\big](y^{-1}x)~\frac{m(x^{-1},y)m(x^{-1}y,y^{-1}x)}{m(x,x^{-1})} \, \mathrm{d}y\\
        &= \big[f \ast' \mathcal W_{\varphi_0}\varphi_0\big](x),
    \end{align*}
    where we used the identity \eqref{eq:some_cocycle_equation}. Furthermore, for $\phi \in \Co_\infty(U)$ and $g \in \Co_1(U)$, we get
    \[\big[\mathcal W_{\varphi_0}^\ast \mathcal W_{\varphi_0} \phi\big](g) = \phi(\mathcal W_{\varphi_0}^\ast \mathcal W_{\varphi_0} g) = \| \varphi_0\|_{\mathcal H}^2 \phi(g).\qedhere\]
\end{proof}

The following proposition is now immediate.

\begin{proposition}~ \label{prop:rangeWaveletTrafo}
Let $\varphi_0 \in \mathcal H_1 \setminus \set{0}$. 
    \begin{itemize}
        \item[(i)] The following equality holds true: $\mathcal W_{\varphi_0}(\Co_\infty(U)) = L^\infty(\Xi) \ast' \mathcal W_{\varphi_0}(\varphi_0)$.
        \item[(ii)] A function $f \in L^\infty(\Xi)$ is contained in $\mathcal W_{\varphi_0}(\Co_\infty(U))$ if and only if $f = \frac{1}{\| \varphi_0\|_{\mathcal H}^2} f \ast' \mathcal W_{\varphi_0}(\varphi_0)$.
        \item[(iii)] The map $f \mapsto f \ast' \mathcal W_{\varphi_0}(\varphi_0)$ is a continuous projection from $L^\infty(\Xi)$ onto $\mathcal W_{\varphi_0}(\Co_\infty(U))$.
    \end{itemize}
\end{proposition}

We now define general coorbit spaces. For $1 \leq p \leq \infty$, we let
\begin{align*}
    \Co_p(U) = \{ \phi \in \mathcal H_\infty: ~\mathcal W_{\varphi_0}(\phi) \in L^p(\Xi)\}.
\end{align*}
We endow this space with the norm $\| \phi\|_{p, \varphi_0} := \| \mathcal W_{\varphi_0}(\phi)\|_{L^p(\Xi)}$. It is not difficult to prove that for $p \in \set{1, 2, \infty}$ the above definition is consistent with the previous definitions (recall that $\Hc$ embeds continuously into $\Hc_{\infty}$). Moreover, these spaces are Banach spaces for all $p$ and similar characterizations regarding the range of the Wavelet function hold true: $\mathcal W_{\varphi_0}(\Co_p(U)) = L^p(\Xi) \ast' \mathcal W_{\varphi_0}(\varphi_0)$. The projection from $L^p(\Xi)$ onto the range of the Wavelet transform is given by the convolution operator. One can further show that the coorbit spaces satisfy the continuous inclusions $\Co_p(U) \subseteq \Co_q(U)$ for $1 \leq p \leq q \leq \infty$, as well as the dualities $\Co_p(U)' \cong \Co_{p'}(U)$, where $1/p + 1/p' = 1$, whenever $1 \leq p < \infty$. Finally, they interpolate under the complex interpolation method as expected. For details, we refer to \cite{Berge2022}. Among all these results, we only want to discuss one explicitly, namely the predual of $\Co_1(U)$, since a proof of this seemingly does not exist spelled out in the literature (even though the result is certainly well-known):

\begin{lemma} \label{lem:predual_of_Co_1}
    Denote
    \begin{align*}
        \Co_{\infty, 0}(U) := \{ f \in \Co_\infty(U): ~\mathcal W_{\varphi_0}(f) \in C_0(\Xi)\},
    \end{align*}
    where $C_0(\Xi)$ denotes the closure of $C_c(\Xi)$ in $L^{\infty}(\Xi)$. This is a closed subspace of $\Co_\infty(U)$. Moreover, under the usual $L^2$-dual pairing, $\Co_{\infty, 0}(U)^\ast \cong \Co_1(U)$ holds.
\end{lemma}

\begin{proof}
    Since $C_0(\Xi)$ is a closed subspace of $L^\infty(\Xi)$, the first statement is obvious. To prove the duality $\Co_{\infty, 0}(U)^\ast \cong \Co_1(U)$, we equivalently prove that $\mathcal W_{\varphi_0}(\Co_{\infty, 0}(U))^\ast \cong \mathcal W_{\varphi_0}(\Co_1(U))$.

    To do so, we first recall that the dual of $C_0(\Xi)$ can be identified with $M(\Xi)$, the space of finite Borel-regular complex measures on $\Xi$, endowed with the total variation norm. It is more convenient for our purposes to make this identification anti-linear, so it formally agrees with the $L^2$-dual pairing:
    \begin{align*}
        \Phi_\mu(f) = \int_\Xi f(x) \, \mathrm{d}\overline{\mu}(x), \quad f \in \mathcal W_{\varphi_0}(\Co_{\infty, 0}(U)).
    \end{align*}
    Since $\mathcal W_{\varphi_0}(\Co_{\infty, 0}(U))$ is a closed subspace of $C_0(\Xi)$, we clearly have $\mathcal W_{\varphi_0}(\Co_{\infty, 0}(U))^\ast \subseteq M(\Xi)$ under the above identification. 
    By Proposition \ref{prop:rangeWaveletTrafo} (ii), we have $f = f \ast' \mathcal W_{\varphi_0}(\varphi_0)$. It follows
    \begin{align*}
        \Phi_\mu(f) &= \int_\Xi f(x) \, \mathrm{d}\overline{\mu}(x)\\
        &= \int_\Xi \frac{1}{\| \varphi_0\|_{\mathcal H}^2} \int_\Xi f(y)\big[W_{\varphi_0}(\varphi_0)\big](y^{-1}x)m(y,y^{-1}x) \, \mathrm{d}y \, \mathrm{d}\overline{\mu}(x)\\
        &= \int_\Xi f(y)\underbrace{\frac{1}{\| \varphi_0\|_{\mathcal H}^2}\int_\Xi \big[W_{\varphi_0}(\varphi_0)\big](y^{-1}x)m(y,y^{-1}x) \, \mathrm{d}\overline{\mu}(x)}_{=: \overline{g(y)}} \, \mathrm{d}y.
    \end{align*}
    Note that interchanging the integrals is permitted here since $\overline{\mu}$ is a finite measure and $W_{\varphi_0}(\varphi_0)$ has $\sigma$-finite support. Clearly, $g \in L^1(\Xi)$. Moreover, $L^1(\Xi) \ast' \mathcal W_{\varphi_0}(\varphi_0) = \mathcal W_{\varphi_0}(\Co_{1}(U))$ by Lemma \ref{lem:Co_1_Banach_space} (ii). Therefore, repeating the above argument with $\mu := g$ (more precisely, $\mu := g\lambda$ with the Haar measure $\lambda$), we can assume $g \in \mathcal W_{\varphi_0}(\Co_1(U))$. This shows that the map $\Phi \from W_{\varphi_0}(\Co_1(U)) \to \mathcal W_{\varphi_0}(\Co_{\infty, 0}(U))^\ast$, $g \mapsto \Phi_g$ is surjective and bounded. Similarly, we also see that
    \[\frac{1}{\| \varphi_0\|_{\mathcal H}^2} \sp{f \ast' \mathcal W_{\varphi_0}(\varphi_0)}{g}_{L^2(\Xi)} = \frac{1}{\| \varphi_0\|_{\mathcal H}^2} \sp{f}{g \ast' \mathcal W_{\varphi_0}(\varphi_0)}_{L^2(\Xi)} = \sp{f}{g}_{L^2(\Xi)}\]
    for all $f \in C_c(\Xi)$. Noting that $f \ast' \mathcal W_{\varphi_0}(\varphi_0) \in \mathcal W_{\varphi_0}(\Co_{\infty, 0}(U))$, this shows that $g \mapsto \Phi_g$ is injective and the lemma follows.
\end{proof}

We end this section by discussing the property of the parity operator with respect to coorbit spaces. Recall that the parity operator $R$ is defined as described in Remark \ref{rem:one}.

\begin{lemma}\label{lem:parity_operator}
    The parity operator $R$ acts continuously on every coorbit space $\Co_p(U)$. If the window function $\varphi_0$ satisfies either $R\varphi_0 = \varphi_0$ or $R\varphi_0 = -\varphi_0$, then $R$ acts isometrically on the coorbit spaces with respect to the norm $\| \cdot \|_{p, \varphi_0}$.
\end{lemma}
We want to emphasize that, by basic spectral theory, there always exists a nontrivial $\varphi_0$ such that $R\varphi_0 = \pm \varphi_0$.

\begin{proof}
    $R$ acts as a unitary operator on $\Co_2(U) = \mathcal H$. Let $\varphi_0 \in \mathcal H_1$. Then
    \begin{align*}
        \mathcal W_{R\varphi_0}(R\varphi_0)(x) = \langle R\varphi_0, U_x R\varphi_0\rangle = \langle \varphi_0, U_{x^{-1}} \varphi_0\rangle,
    \end{align*}
    which implies $\mathcal W_{R\varphi_0}(R\varphi_0) \in L^1(\Xi)$ with $\| \mathcal W_{R\varphi_0}(R\varphi_0)\|_{L^1} = \| \mathcal W_{\varphi_0}(\varphi_0)\|_{L^1} $ and therefore $R\varphi_0 \in \mathcal H_1$. Now, if $\varphi_0$ satisfies $R\varphi_0 = \pm \varphi_0$, then for every $f \in \Co_2(U)$ we have
    \begin{align*}
        \mathcal W_{\varphi_0} (Rf)(x) = \pm \mathcal W_{\varphi_0}(f)(x^{-1}),
    \end{align*}
    which is in $L^p(\Xi)$ provided $\mathcal W_{\varphi_0}(f)$ is in $L^p(\Xi)$. By density, $R$ acts isometrically on $\Co_p(U)$ for every $1 \leq p < \infty$. On $\Co_\infty(U)$, the statement follows from the one on $\Co_1(U)$ by duality.
\end{proof}

From the formulas written down in the previous proof, we also obtain:

\begin{corollary}
    Let $\varphi_0 \in \mathcal H_1 \setminus \set{0}$ be arbitrary. Then, $\mathcal W_{R \varphi_0} R \mathcal W_{\varphi_0}^\ast$ maps $\mathcal W_{\varphi_0}(\Co_p(U))$ isometrically onto $\mathcal W_{R\varphi_0}(\Co_p(U))$.
\end{corollary}

\section{Operator algebras and standard representations} \label{sec:operator_algebras}
We say that the projective representation $(U, m)$ is \emph{standard} provided we can factorize $\Xi = G \times \widehat{G}$, where $G$ is a locally compact abelian group with neutral element $e$, and there exists a continuous function $a: G \times \widehat{G} \to S^1$ with $a(e, 1) = 1$ and $a(x^{-1},\nu^{-1}) = a(x, \nu)$ such that we can write
\begin{align*}
    m(\xi,\eta) = m\big((x, \nu), (y, \mu)\big) = \frac{a(x, \nu)a(y, \mu)}{a(xy, \nu \mu)} \overline{\mu(x)}, \quad \xi = (x, \nu), \eta = (y, \mu) \in \Xi = G \times \widehat{G}.
\end{align*}
During this section, we always assume that our representation is standard. In this case, we can realize the representation $U$ on $\mathcal H = L^2(G)$ as $U_{(x, \nu)}f(y) = a(x, \nu) \nu(y) f(x^{-1}y)$. Further, the coorbit spaces $\Co_p(U)$ agree with the modulation spaces $M^p(G)$. Indeed, for $f \in L^2(G)$ we get
\begin{align*}
    |\langle f, U_{(x, \nu)}f\rangle| = \abs{\int_G f(y) \overline{f(x^{-1}y)\nu(y)} \, \mathrm{d}y},
\end{align*}
and $(x,\nu) \mapsto \langle f, U_{(x, \nu)}f\rangle \in L^1(\Xi)$ is one of many equivalent defining properties of the modulation space $M^1(G)$, which is the same as the Feichtinger algebra $S_0(G)$, cf.\ \cite{Feichtinger81, jakobsen18}.

For $A \in \mathcal L(M^p(G))$ and $\xi \in \Xi$ we now define
\[\alpha_{\xi}(A) := U_{\xi}AU_{\xi}^{-1}.\]
Similarly, if $\xi,\eta \in \Xi$ and $x,y \in G$, we define $\big[\alpha_{\xi}(f)\big](\eta) := f(\xi^{-1}\eta)$ for functions $f \from \Xi \to \C$ and $\big[\alpha_x(f)\big](y) := f(x^{-1}y)$ for functions $f \from G \to \C$.

We will consider the following operator subalgebras of $\mathcal L(M^p(G))$:
\begin{definition}
    Let $1 \leq p \leq \infty$. Then we set
    \begin{align*}
        \mathcal C_{1,0}^p(G) &:= \{ A \in \mathcal L(M^p(G)): ~\|\alpha_{(x, 1)}(A) - A\|_{op} \to 0 \text{ as } x \to e\},\\
    \mathcal C_{0,1}^p(G) &:= \{ A \in \mathcal L(M^p(G)): ~\|\alpha_{(e, \nu)}(A) - A\|_{op} \to 0 \text{ as } \nu \to 1\},\\
    \mathcal C_1^p(G) &:= \mathcal C_{1,0}^p(G) \cap \mathcal C_{0,1}^p(G)\\
    &= \{ A \in \mathcal L(M^p(G)): ~\| \alpha_{(x,\nu)}(A) - A\|_{op} \to 0 \text{ as } (x, \nu) \to (e,1)\}.
    \end{align*}
\end{definition}

Our main interest will be to compare these algebras with the following class of operators:

\begin{definition} \label{def:BDO}
    Let $G$ be an lca group and $K \subseteq G$ be compact. An operator $A \in \Lc(M^p(G))$ is called a band operator of band-width at most $K$ if for all $H \subseteq G$ and all $f \in M^p(G)$ with $\supp f \subseteq H$ we have
    \begin{equation} \label{eq:BO_condition}
        \supp(Af) \subseteq HK = \set{hk : h \in H, k \in K}.
    \end{equation}
    The set of all band operators of band-width at most $K$ will be denoted by $\BO_K^p(G)$. The union $\bigcup\limits_{K \subseteq G \text{ compact}} \BO_K^p(G)$ is denoted by $\BO^p(G)$ and called the set of band operators. The closure of $\BO(G)$ with respect to the norm topology is denoted by $\BDO^p(G)$ and its elements are called band-dominated operators.
\end{definition}

\newpage
\begin{remark}~
    \begin{itemize}
        \item[(a)] Since $M^2(G) = L^2(G)$, multiplying by characteristic functions $\1_K$ and $\1_{G \setminus HK}$ is well-defined in $M^2(G)$ and therefore \eqref{eq:BO_condition} is equivalent to $M_{\1_{G \setminus HK}}AM_{\1_H} = 0$ in case $p = 2$. In this form, Definition \ref{def:BDO} is a metric-free version of \cite[Definition 3.2]{HaggerSeifert}. More precisely, if we additionally assume that $G$ is second countable, then it is properly metrizable (in the sense that every closed ball is compact) and our definition of band operators becomes equivalent to the usual definition (e.g.~\cite[Definition 3.2]{HaggerSeifert} or \cite[Definition 2.1.5]{Rabinovich_Roch_Silbermann2004}). The compact set $K$ can be chosen as a ball of a certain radius centered at $e$ and the smallest possible radius is then the band width of $A$ in the sense of \cite[Definition 3.2]{HaggerSeifert}. For instance, if $G = \Z$ and $A \in \BO_K(\Z)$, this means that only the diagonals $k \in K$ in the matrix representation of $A$ can be non-zero, where $k = 0$ represents the main diagonal and so on.
        \item[(b)] For discrete groups $G$ the set $\BO^p(G)$ is independent of $p$ in the sense that each $T \in \BO^p(G)$ can be represented as a matrix for which only finitely many diagonals are non-zero. As diagonal matrices are in $\BO^p(G)$ if and only if the diagonal is bounded, this shows the independence of $p$. However, this is not true for all groups. For example, if $G$ is a compact group, then every bounded linear operator is a band operator as we may choose $K = G$. But $\Lc(M^p(G))$ is not independent of $p$ in general.
        \item[(c)] The sets of band operators of different band-width are partially ordered by inclusion. That is, if $A$ is of band-width at most $K$ and $K \subseteq K'$, then it is also of band-width at most $K'$. It is therefore clear that in the definition of $\BO^p(G)$ it suffices to take the union over those compact sets $K \subseteq G$ that contain $e$ and satisfy $K = K^{-1}$. This will be useful for a few arguments.
    \end{itemize}
\end{remark}

\begin{remark}
    In the following discussions, we restrict ourselves to $p < \infty$. Indeed, on the modulation space $M^\infty(G)$, the arguments we present do not work and some of the statements made in the following would simply be wrong in this case. Nevertheless, if we denote by $M^{\infty, 0}(G)$ the coorbit space $\operatorname{Co}_{\infty, 0}(U)$ as introduced in Lemma \ref{lem:predual_of_Co_1}, all the statements and proofs made in the rest of this section carry over to operators on $M^{\infty, 0}(G)$. Essentially for the reason of notational simplicity, we will not keep track of this in the following, but the interested reader should keep in mind that there is no difficulty in working with this space instead.
\end{remark}

We will first show some algebraic properties of $\BDO^p(G)$.

\begin{proposition} \label{prop:BDO_algebraic_properties}
    Let $G$ be an lca group. Then $\BO^p(G)$ is an algebra and $\BDO^p(G)$ is a Banach algebra. Moreover, $(\BDO^p(G))^* = \BDO^q(G)$ via the usual identification of $M^p(G)^*$ and $M^q(G)$ for $p \in [1,\infty)$ and $\frac{1}{p} + \frac{1}{q} = 1$. In particular, $\BDO^2(G)$ is a $C^*$-algebra.
\end{proposition}

\begin{proof}
    As band operators are partially ordered by inclusion, it is clear that they form a vector space. Now let $K_1,K_2 \subseteq G$ be compact and $A_j \in \BO_{K_j}^p(G)$ for $j = 1,2$. For $H \subseteq G$ and $f \in M^p(G)$ with $\supp f \subseteq H$ we get $\supp(A_1f) \in HK_1$ and thus $\supp(A_2A_1f) \in HK_1K_2$. It follows that $A_2A_1 \in \BO_{K_1K_2}^p(G)$. Therefore, $\BO^p(G)$ is an algebra. Taking the closure thus yields a Banach algebra. Finally, assume $\frac{1}{p} + \frac{1}{q} = 1$, let $K \subseteq G$ be compact and let $A \in \BO_K^p(G)$. Note that $\big(G \setminus (HK^{-1})\big)K \subseteq G \setminus H$ for every $H \subseteq G$. Using the duality between $M^p(G)$ and $M^q(G)$ (cf.\ \cite{Feichtinger2002}), we thus obtain
    \[0 = \sp{Af}{g} = \sp{f}{A^*g}\]
    for $f \in M^p(G)$, $g \in M^q(G)$ with $\supp f \subseteq G \setminus (HK^{-1})$ and $\supp g \in H$. In particular, $\supp(A^*g) \subseteq HK^{-1}$. Therefore, $A^* \in \BO_{K^{-1}}^p(G)$. This implies $(\BDO^p(G))^* = \BDO^q(G)$, which also shows that $\BDO^2(G)$ is a $C^*$-algebra.
\end{proof}

In the following, we will make use of the following notation, which plays a key role in quantum harmonic analysis: For $A \in \mathcal L(M^p(G))$ and a finite, possibly complex, Radon measure $\mu$ on $\Xi$ we define
\begin{align*}
    \mu \ast A := \int_{\Xi} \alpha_z(A) \, \mathrm{d}\mu(z)
\end{align*}
with the usual interpretation if $\mu \in L^1(\Xi)$. We refer to \cite{Fulsche_Galke2025, Luef_Skrettingland2021, werner84} for properties of this operation. We note that this integral always exists in the following sense: Since
\begin{align*}
    \mathcal L(M^p(G)) \cong (M^q(G) \widehat{\otimes}_\pi M^p(G))^\ast, 
\end{align*}
using that $M^q(G)^\ast \cong M^p(G)$ (respectively, $M^1(G) \cong (M^{\infty, 0}(G))^\ast$; see Lemma \ref{lem:predual_of_Co_1}), one can see that $\mathcal L(M^p(G))$ is always a dual space. Hence, $\mu \ast A$ is always defined as a Gelfand integral, cf.~\cite[Section 3]{Fulsche_Galke2025}. A key property that we will make use of is the fact that $\| \mu \ast A\| \leq \| \mu\| \cdot \| A\|$, where $\| \mu\|$ refers to the total variation norm of $\mu$.

In order to compare $\mathcal C_1^p(G)$ with the class of band-dominated operators, we will make use of the following proposition, which is also of independent interest:

\begin{proposition}\label{prop:charC1}
    For $1 \leq p < \infty$ we have
    \begin{align*}
        \mathcal L(M^p(G)) &= (\delta_0 \otimes \delta_0) \ast \mathcal L(M^p(G)),\\
        C_{1,0}^p(G) &= (L^1(G) \otimes \delta_0) \ast \mathcal L(M^p(G)),\\
        C_{0,1}^p(G) &= (\delta_0 \otimes L^1(\widehat{G})) \ast \mathcal L(M^p(G)),\\
        C_1^p(G) &= L^1(G \times \widehat{G}) \ast \mathcal L(M^p(G)).
    \end{align*}
\end{proposition}

In each case, $\delta_0$ denotes the Dirac measure corresponding to the identity of the corresponding group.

\begin{proof}
    For the first equality there is of course nothing to prove; it is mainly stated to point out the analogy with the other three cases. The fourth identity is standard in quantum harmonic analysis. The second and third equality are of course proven similarly. We only prove the equality for $C_{1,0}^p(G)$, the equality for $C_{0,1}^p(G)$ follows in the same way.

    First of all, for $f \in L^1(G)$ and $A \in\mathcal L(M^p(G))$ we have
    \begin{align*}
        \| \alpha_{(x,1)}((f \otimes \delta_0) \ast A) - (f \otimes \delta_0) \ast A\| &= \| ((\alpha_x(f) - f) \otimes \delta_0) \ast A\|\\
        &\leq \| \alpha_x(f) - f\|_{L^1(G)} \| A\|\\
        &\to 0
    \end{align*}
    as $x \to e$. This proves $(L^1(G) \otimes \delta_0) \ast A \in C_{1,0}^p(G)$. Now let $A \in C_{1,0}^p(G)$. If $(g_\gamma)_{\gamma \in \Gamma}$ is a bounded approximate unit of $L^1(G)$, then $(g_\gamma \otimes \delta_0) \ast A \to A$ in operator norm by standard arguments. Hence, $A \in \overline{(L^1(G) \otimes \delta_0) \ast \mathcal L(M^p(G))}$. So far, we have proven that
    \begin{align*}
        C_{1,0}^p(G) = \overline{(L^1(G) \otimes \delta_0) \ast \mathcal L(M^p(G))}.
    \end{align*}
    The Cohen-Hewitt factorization theorem now shows that the closure on the right-hand side is actually redundant.
\end{proof}
Let $A \in \mathcal L(M^p(G))$. Then, we have the following continuous maps:
\begin{align*}
    S_0(G) = M^1(G) \hookrightarrow M^p(G) \overset{A}{\longrightarrow} M^p(G) \hookrightarrow M^\infty(G) = S_0'(G).
\end{align*}
Hence, we can consider $A$ as a bounded linear operator from $S_0(G)$ to $S_0'(G)$. By the outer kernel theorem for $S_0(G)$ (cf.\ \cite[Theorem 9.3]{jakobsen18}), every bounded linear operator from $S_0(G)$ to $S_0'(G)$ is represented by an integral kernel $k_A \in S_0'(G \times G)$, given by the relation:
\begin{align*}
    \langle A \varphi, \psi\rangle_{L^2(G)} = \langle k_A, \overline{\varphi} \otimes \psi\rangle_{L^2(G \times G)}, \quad \varphi, \psi \in S_0(G).
\end{align*}
Here, we chose the convention that the inner product is antilinear in the second entry,
\begin{align*}
    \langle \varphi, \psi\rangle_{L^2(G)} = \int_G \varphi(x) \overline{\psi(x)} \, \mathrm{d}x, \quad \varphi, \psi \in S_0(G),
\end{align*}
with the appropriate extension to $\varphi \in S_0'(G)$. Furthermore, the tensor product $\varphi \otimes \psi$ is bilinear, that is, $\varphi \otimes \psi \in S_0(G \times G)$ is defined for $\varphi, \psi \in S_0(G)$ by $\varphi \otimes \psi(x, y) = \varphi(x) \psi(y)$. 

We will need the following lemma:

\begin{lemma}
    Let $1 \leq p \leq \infty$ and $A \in \mathcal L(M^p(G))$ and $g \in L^1(\widehat{G})$. Denote by $k_A \in S_0'(G \times G)$ the integral kernel of $A$. Then, the integral kernel of $(\delta_0 \otimes g) \ast A$ is given by $h_g \cdot k_A$. Here, the function $h_g \from G \times G \to \C$ is defined by
    \begin{align*}
        h_g(x,y) = \widehat{g}(x^{-1}y).
    \end{align*}
\end{lemma}

\begin{proof}
Let $\varphi, \psi \in S_0(G)$. For $g \in \mathcal S_0(\widehat{G})$ we have
\begin{align*}
    \langle ((\delta_0 \otimes g) \ast A) \varphi, \psi\rangle_{L^2(G)} &= \int_{\widehat{G}}g(\nu) \langle AU_{(e, \nu^{-1})}\varphi, U_{(e, \nu^{-1})}\psi\rangle \, \mathrm{d}\nu\\
    &= \int_{\widehat{G}} g(\nu) \langle k_A, \overline{U_{(e,\nu^{-1})}\varphi}\otimes U_{(e,\nu^{-1})}\psi \rangle \, \mathrm{d}\nu.
\end{align*}
We note that $g \in L^1(\widehat{G})$. Furthermore, the function $\nu \mapsto \overline{U_{(e,\nu^{-1})}\varphi} \otimes U_{(e,\nu^{-1})}\psi$ is bounded and continuous in $S_0(G \times G)$. Hence, the integral
\begin{align*}
    \int_{\widehat{G}} g(\nu) \overline{U_{(e,\nu^{-1})}\varphi}\otimes U_{(e,\nu^{-1})}\psi \, \mathrm{d}\nu
\end{align*}
exists, say, as a Gelfand-Pettis-integral in $S_0(G \times G)$, and it satisfies pointwise:
\begin{align*}
    \left [\int_{\widehat{G}} g(\nu) \overline{U_{(e,\nu^{-1})}\varphi}\otimes U_{(e,\nu^{-1})}\psi \, \mathrm{d}\nu \right ](x,y) &= \int_{\widehat{G}} g(\nu) \nu(x) \overline{\varphi(x)} \nu^{-1}(y) \psi(y) \, \mathrm{d}\nu\\
    &= \overline{\varphi(x)} \psi(y) \int_{\widehat{G}} g(\nu) \nu(y^{-1}x) \, \mathrm{d}\nu\\
    &= \overline{\varphi(x)} \psi(y) \widehat{g}(x^{-1}y).
\end{align*}
Therefore, we see that (by basic properties of the Gelfand--Pettis-integral, pairing with $k_A$ can be pulled out of the integral): 
\begin{align*}
    \langle (\delta_0 \otimes g) \ast A \varphi, \psi\rangle &= \int_{\widehat{G}} g(\nu) \langle k_A, \overline{U_{(e, \nu^{-1})}\varphi}\otimes U_{(e, \nu^{-1})}\psi \rangle \, \mathrm{d}\nu\\
    &= \langle k_A, \overline{h_g} \overline{\varphi} \otimes \psi\rangle\\
    &= \langle h_g k_A, \overline{\varphi} \otimes \psi\rangle,
\end{align*}
which concludes the proof.
\end{proof}

We will also make use of the following fact:

\begin{lemma}\label{lemma:support}
    Let $K \subseteq G$ be compact and $A \in \mathcal L(M^p(G))$, where $1 \leq p \leq \infty$. Then, we have $A \in \operatorname{BO}_K^p(G)$ if and only if the integral kernel of $A$ satisfies
    \begin{align*}
        \supp(k_A) \subseteq \{ (x, xy) \in G \times G: ~x \in G, y\in K\}.
    \end{align*}
\end{lemma}

\begin{proof}
Let $A \in \BO_K^p(G)$, $x,y \in G$ and assume that $x^{-1}y \notin K$. Then we may choose open sets $U,V \subseteq G$ such that $x \in U$, $y \in V$ and $UK \cap V = \emptyset$. To see this, first observe that, since every lca group is regular, there exist open sets $\tilde{U},V \subseteq G$ such that $xK \in \tilde{U}$, $y \in V$ and $U \cap \overline{V} = \emptyset$. So assume that $UK \cap V \neq \emptyset$ for all open sets $U \subseteq G$ with $x \in U$. This means that for every $U \subseteq G$ with $x \in U$ there exist $z_U \in U$ and $k_U \in K$ such that $z_Uk_U \in V$. With respect to the partial order $\supseteq$, the net $(z_U)$ converges to $x$ and by compactness there is a subnet of $(k_U)$ that converges to some $k \in K$. But this implies $xk \in \overline{V}$, which is impossible.

Now, for every $\varphi,\psi \in S_0(G)$ with $\supp\varphi \subseteq U$ and $\supp\psi \subseteq V$ we have $\supp(A\varphi) \subseteq UK$ and hence $\sp{A\varphi}{\psi} = 0$. This shows that $(x,y) \notin \supp (k_A)$. It follows
\[\supp (k_A) \subseteq \set{(x,y) \in G \times G: x^{-1}y \in K}. \]
On the other hand, when the kernel of $A$ satisfies
\begin{align*}
     \supp(k_A) \subset \{ (x, xy) \in G \times G: ~x \in G, y \in K\}
\end{align*}
and $\varphi, \psi \in S_0(G)$ are such that $\supp(\varphi) \subseteq H$, $\supp(\psi) \subseteq G \setminus HK$, then
\begin{align*}
    H \times (G \setminus HK) \cap \supp(k_A) = \emptyset,
\end{align*}
which implies
\begin{align*}
    \langle A \varphi, \psi\rangle = \langle k_A, \overline{\varphi} \otimes \psi\rangle = 0.
\end{align*}
Since $\psi$ was arbitrary with $\supp(\psi) \subseteq G \setminus HK$, this shows $\supp(A\varphi) \subseteq HK$. Hence, $A$ is a band operator with band-width at most $K$.
\end{proof}

\begin{theorem}
    Let $1 \leq p \leq \infty$. Then the following equality holds:
    \begin{align*}
        \mathcal C_{0,1}^p(G) = \operatorname{BDO}^p(G).
    \end{align*}
\end{theorem}

\begin{proof}
    We first discuss the inclusion ``$\subseteq$''. Let $A \in \mathcal C_{0,1}^p(G)$. By Proposition \ref{prop:charC1}, we may assume that we have $A = (\delta_0 \otimes g) \ast B$, with $g \in L^1(\widehat{G})$ and $B \in \Lc(M^p(G))$. By density, we may also assume that $\widehat{g}$ is compactly supported. In this case, the integral kernel of $A$ is given by $k_A = h_g k_B$, where $k_B$ is the kernel of $B$. Since $\supp(h_g) \subseteq \{ (x,xy) \in G \times G: ~x \in G, y \in \supp(\widehat{g})\}$, we see that
    \begin{align*}
        \supp(k_A) \subseteq \{ (x,xy) \in G \times G: ~x \in G, y \in K\}
    \end{align*}
    for $K := \supp(\widehat{g})$ compact. Therefore, Lemma \ref{lemma:support} shows that $A \in \operatorname{BO}_K^p(G)$. This shows the first inclusion.

    For the inclusion ``$\supseteq$'' let $K \subseteq G$ be compact and $A \in \operatorname{BO}_K^p(G)$. Let $g \in L^1(\widehat{G})$ be such that $\widehat{g} = 1$ on $K$ (this always exists, cf.\ \cite[Prop.\ 5.1.9]{Reiter_Stegeman2000}). Then, by Lemma \ref{lemma:support}, the kernel of $A$ satisfies
    \begin{align*}
        k_A = h_g k_A
    \end{align*}
    such that $A = (\delta_0 \otimes g) \ast A$, proving that $A \in \mathcal C_{0,1}^p(G)$. This finishes the proof.
\end{proof}

By Pontryagin duality, the following corollary is immediate.

\begin{corollary}\label{corollary:equality_algebras}
    Let $1 \leq p \leq \infty$. The following statements hold:
    \begin{align*}
        \mathcal C_{1,0}^p(G) &= \mathcal F^{-1} \operatorname{BDO}^p(\widehat{G}) \mathcal F,\\
        \mathcal C_1^p(G) &= \operatorname{BDO}^p(G) \cap (\mathcal F^{-1} \operatorname{BDO}^p(\widehat{G}) \mathcal F).
    \end{align*}
\end{corollary}

As an application of the previous results, we discuss the characterization of compact operators on $M^p(G)$. We restrict the discussion to the reflexive range, that is, $1 < p < \infty$. Nevertheless, suitable modifications allow for the application of the present methods also for the end-point cases $p \in \set{1, \infty}$ (see \cite{Lindner2006} or \cite{Rabinovich_Roch_Silbermann2004} for $G = \mathbb Z^d$ or \cite{Fulsche2024} for a discussion of the case $G = \mathbb R^d$).

For formulating the theorem, we note that for each $B \in \mathcal C_1^p(G)$, the (norm) continuous map
\begin{align*}
    G \times \widehat{G} \to \mathcal C_1^p(G), \quad (x, \xi) \mapsto \alpha_{(x, \xi)}(B)
\end{align*}
extends to a strongly continuous map
\begin{align*}
    \mathcal M(\BUC(G \times \widehat{G})) \to \mathcal C_1^p(G), \quad z \mapsto \alpha_z(B).
\end{align*}
Here, $\mathcal M(\BUC(G \times \widehat{G}))$ denotes the maximal ideal space of the bounded uniformly continuous functions on $G \times \widehat{G}$, interpreted as a compactification of $G \times \widehat{G}$. The previous statement can be proven analogously to the Hilbert space case $p = 2$ just as in \cite{Fulsche_Luef_Werner2024}.

Now, imitating the proof of \cite[Lemma 3.8]{Fulsche_Luef_Werner2024}, we obtain the following result. Here, we write $\partial \Xi := \mathcal M(\BUC(G \times \widehat{G})) \setminus (G \times \widehat{G})$.

\begin{theorem} \label{thm:compactness}
    Let $1 < p < \infty$ and $B\in \mathcal L(M^p(G))$. Then $B$ is compact if and only if $B \in \mathcal C_1^p(G) = \BDO^p(G) \cap (\Fc^{-1} \BDO^p(\widehat{G}) \Fc)$ and $\alpha_z(B) = 0$ for every $z \in \partial \Xi$.
\end{theorem}

Although not written down in the literature so far, the only fact about the previous theorem that does not follow from the usual methods of quantum harmonic analysis (cf.~\cite{werner84, Fulsche_Galke2025, Fulsche_Luef_Werner2024}) is that we can use $\operatorname{BDO}^p(G) \cap (\mathcal F^{-1} \operatorname{BDO}^p(\widehat{G}) \mathcal F)$ instead of $\mathcal C_1^p(G)$. Since, besides the application of Corollary \ref{corollary:equality_algebras}, the proof is an imitation of the results in \cite{werner84} with the tools described in \cite{Fulsche_Galke2025}, we omit the proof. Instead, we focus on two particular instances of this result:

\begin{corollary}\label{cor:compactness_discrete_grp}~
    \begin{itemize}
        \item[(a)] Assume that $G$ is discrete. Then $B \in \mathcal L(M^p(G))$ is compact if and only if $B \in \operatorname{BDO}^p(G)$ and $\alpha_x(B) = 0$ for every $x \in \partial G$.
        \item[(b)] Assume that $G$ is compact. Then $B \in \mathcal L(M^p(G))$ is compact if and only if $B \in \mathcal F^{-1} \operatorname{BDO}^p(\widehat{G}) \mathcal F$ and $\alpha_\xi(B) = 0$ for every $\xi \in \partial \widehat{G}$. 
    \end{itemize}
\end{corollary}

Here, we use $\alpha_x(B)$ and $\alpha_\xi(B)$ as a shorthand for the strong limits $\lim\limits_{x_\gamma \to x} \alpha_{(x_\gamma,1)}(B)$ and $\lim\limits_{\xi_\gamma \to \xi} \alpha_{(e,\xi_\gamma)}(B)$, respectively. Moreover, by $\partial G$ we technically mean the Stone--\v{C}ech boundary. However, since $G$ is discrete in this instance, $\partial G$ is isomorphic to $\mathcal M(\BUC(G)) \setminus G$. In fact, for this corollary specifically, any compactification of $G$ can be used. Indeed, one might as well reformulate ``$\alpha_x(B) = 0$ for every $x \in \partial G$'' to ``$\alpha_{x_{\gamma}}(B) \to 0$ for any net $(x_{\gamma})_{\gamma \in \Gamma}$ that has no convergent subnet in $G$'' and not use any compactification at all. For a brief discussion about different compactifications we refer to \cite[Section 5]{HaggerSeifert}.

For the proof we need the following lemma.

\begin{lemma} \label{lem:product_boundary}
    Let $G$ be discrete and let $(x_\gamma, \xi_\gamma)_{\gamma \in \Gamma}$ be a net in $\Xi = G \times \widehat{G}$ that converges to some point $z \in \partial \Xi$. Then the following assertions hold.
    \begin{itemize}
        \item[(i)] There is a $\xi \in \widehat{G}$ such that $\xi_\gamma \to \xi$.
        \item[(ii)] The net $(x_\gamma, \xi)_{\gamma \in \Gamma}$ also converges to some point $z' \in \partial \Xi$.
        \item[(iii)] Let $B \in \mathcal C_1(M^p(G))$ and $z,z' \in \partial\Xi$ be as above. Then $\alpha_z(B) = \alpha_{z'}(B)$.
    \end{itemize}
\end{lemma}

\begin{proof}~
    \begin{itemize}
       \item[(i)] Note that we can embed $1 \otimes C(\widehat{G})$ into $\operatorname{BUC}(\Xi)$. In particular, $(\psi(\xi_{\gamma}))_{\gamma \in \Gamma}$ converges for all $\psi \in C(\widehat{G})$. Since $\widehat{G}$ is compact, $(\xi_\gamma)_{\gamma \in \Gamma}$ needs to converge to some $\xi \in \widehat{G}$.
        \item[(ii)] A net $(x_\gamma, \xi_\gamma)_{\gamma \in \Gamma}$ converges in $\mathcal M(\BUC(\Xi))$ if and only if $(f(x_\gamma, \xi_\gamma))_{\gamma \in \Gamma}$ converges for every $f \in \BUC(\Xi)$. Let $f \in \BUC(\Xi)$ and consider the function
        \[g \from \Xi \to \C, \quad g(x',\xi') := f(x',\xi).\]
        Clearly, $g \in \BUC(\Xi)$ and therefore $(g(x_\gamma, \xi_\gamma))_{\gamma \in \Gamma} = (f(x_\gamma,\xi))_{\gamma \in \Gamma}$ converges. This shows that $(x_\gamma, \xi)_{\gamma \in \Gamma}$ converges to some $z' \in \mathcal M(\BUC(\Xi))$. Since $f(x_\gamma, \xi) \to 0$ when $f \in C_0(\Xi)$, we necessarily have $z' \in \partial \Xi$.
        \item[(iii)] Note that $\alpha_{(x_\gamma,\xi_\gamma)}(B) = \alpha_{(x_\gamma,\xi)}(\alpha_{e,\xi^{-1}\xi_\gamma}(B))$. The left-hand side converges to $\alpha_z(B)$, whereas the right-hand side converges to $\alpha_{z'}(B)$.\qedhere
    \end{itemize}
\end{proof}

\begin{proof}[Proof of Corollary \ref{cor:compactness_discrete_grp}]
    We only prove (a) as (b) is just the dual result. If $G$ is discrete, then $\widehat{G}$ is compact and therefore $\BDO(\widehat{G}) = \Lc(M^p(\widehat{G}))$. Therefore, $\mathcal C_1^p(G) = \BDO^p(G)$ by Corollary \ref{corollary:equality_algebras}. Let $B \in \BDO^p(G)$. According to Theorem \ref{thm:compactness}, we only need to show that if $\alpha_x(B) = 0$ for every $x \in \partial G$, then $\alpha_z(B) = 0$ for all $z \in \partial \Xi$. So let $(x_\gamma, \xi_\gamma)_{\gamma \in \Gamma}$ be a net in $G \times \widehat{G}$ that converges to some $z \in \partial \Xi$. Lemma \ref{lem:product_boundary} implies that there is a $\xi \in \widehat{G}$ and a $z' \in \partial\Xi$ such that $\xi_\gamma \to \xi$, $(x_\gamma,\xi) \to z'$ and $\alpha_z(B) = \alpha_{z'}(B)$. By taking a subnet if necessary, we may also assume that $x_\gamma \to x \in \partial G$ for some $x \in \partial G$. Now, if $\alpha_{(x_\gamma,1)}(B) \to 0$ as $x_\gamma \to x$, then also
    \[\alpha_{z'}(B) = \lim\limits_{x_\gamma \to x} \alpha_{x_\gamma,\xi}(B) = \lim\limits_{x_\gamma \to x} \alpha_{(e,\xi)}(\alpha_{(x_\gamma,1)}(B)) = 0.\]
    It follows that $\alpha_z(B) = 0$ for all $z \in \partial \Xi$.
\end{proof}

Combining Corollary \ref{cor:compactness_discrete_grp} with the following well-known fact yields the standard compactness characterization for operators acting on $\ell^p$-spaces over discrete abelian groups. It is worth mentioning that \cite{Spakula_Willett} does not assume that $G$ is an abelian group. However, the assumptions of \cite{Spakula_Willett} imply that $G$ is countable, which we do not assume.

\begin{proposition}~
    \begin{itemize}
        \item[(a)] Let $G$ be discrete. Then $M^p(G) = \ell^p(G)$.
        \item[(b)] Let $G$ be compact. Then $M^p(G) = \mathcal F^{-1}\big(\ell^p(\widehat{G})\big)$.
    \end{itemize}
\end{proposition}

\begin{proof}
    We only prove the first result as the second one is an obvious consequence of the first one combined with the fact that $\mathcal F \big(M^p(G)\big) = M^p(\widehat{G})$.

    When $G$ is a discrete group, we can choose the window function as
    \[\varphi_0(x) := \begin{cases}
            1 & \text{for } x = e,\\
            0 & \text{for } x \neq e.
        \end{cases}\]    
    Clearly, 
    \begin{align*}
        \mathcal W_{\varphi_0}(\varphi_0)(x, \xi) = \langle \varphi_0, U_{(x, \xi)}\varphi_0\rangle = \begin{cases}
            1 & \text{for } x = e,\\
            0 & \text{for } x \neq e.
        \end{cases}
    \end{align*}
    Therefore,
    \begin{align*}
        \int_{\widehat{G}}\int_G \abs{\mathcal W_{\varphi_0}(\varphi_0)(x, \xi)} \, \mathrm{d}x \, \mathrm{d}\xi = \int_{\widehat{G}} 1 \, \mathrm{d}\xi < \infty
    \end{align*}
    because $\widehat{G}$ is compact. It follows $\varphi_0 \in \mathcal H_1$. Now, for any function $f: G \to \mathbb C$ we have that
    \begin{align*}
        \int_{\widehat{G}} \int_G |\mathcal W_{\varphi_0}(f)(x, \xi)|^p \, \mathrm{d}x \, \mathrm{d}\xi &= \int_{\widehat{G}} \int_G |f(x)|^p \, \mathrm{d}x \, \mathrm{d}\xi.
    \end{align*}
    In particular, the $M^p(G)$-norm is equivalent to the $L^p(G)$-norm, proving the statement.
\end{proof}

We arrive at the following classical theorem for discrete abelian groups:

\begin{theorem} \label{thm:matrix_case}
    Let $B \in \mathcal L(\ell^p(G))$, where $G$ is a discrete abelian group, and let $(B_{j,k})_{j,k \in G}$ denote the matrix coefficients of $B$. Then $B$ is compact if and only if $B \in \operatorname{BDO}^p(G)$ and $B_{jm_\gamma,km_\gamma} \to 0$ as $m_{\gamma} \to \partial G$ for every pair $(j,k) \in G \times G$.
\end{theorem}

At first sight, the condition $B_{jm_\gamma,km_\gamma} \to 0$ might seem strange and one would probably expect $B_{j + m_\gamma,k + m_\gamma} \to 0$ instead. Note that this is only due to our convention to write lca groups multiplicatively.

Let us now consider some simple examples of operators on $L^2(G)$.
\begin{example}~
    \begin{itemize}
        \item[(a)] As mentioned in the introduction, it appears that the limit operator method does not work as expected for multiplication operators on $\R$. We want to revisit this example to see what is going wrong. Let $G$ be an arbitrary lca group and let us write $\Xi := G \times \widehat{G}$ as before. Consider $f \in C_0(G)$ such that $\widehat{f} \in L^1(\widehat{G})$. Then clearly the multiplication operator $M_f$ is in $\BDO^2(G)$. Moreover,
        \[C_{\widehat f}g(\nu) := \Fc M_f \Fc^{-1}g(\nu) = \int_{\widehat{G}} \widehat{f}(\mu^{-1}\nu)g(\mu) \, \mathrm{d}\mu\]
        for all $g \in L^2(\widehat{G})$ and $\nu \in \widehat{G}$. As we assumed $\widehat{f} \in L^1(\widehat{G})$, Lemma \ref{lemma:support} shows that $\Fc M_f \Fc^{-1} \in \BDO^p(\widehat{G})$. Now let us consider the limit operators of $M_f$. As expected,
        \[\alpha_{(x,\nu)}(M_f) = M_{f(x^{-1} \cdot)}\]
        and so if we fix $\nu \in \widehat{G}$, then $\alpha_{(x,\nu)}(M_f) \to 0$ as $(x,\nu) \to \partial\Xi$. However, if we fix $x \in G$, say $x = e$, then $\alpha_{(e,\nu)}(M_f)$ does not converge to $0$ as $(e,\nu) \to \partial\Xi$. Therefore Theorem \ref{thm:compactness} implies that $M_f$ is not compact unless $G$ is discrete (because in the latter case there is no net $(\nu_{\gamma})_{\gamma \in \Gamma}$ such that $(e,\nu_{\gamma}) \to \partial\Xi$, see Corollary \ref{cor:compactness_discrete_grp}). Hence we get the expected result.
        \item[(b)] Let $\varphi \in C_0(G)$ and $\psi \in L^1(G)$ such that $\widehat{\varphi} \in L^1(\widehat{G})$ and $\widehat{\psi} \in C_0(\widehat{G})$. Consider the product of the multiplication operator $M_{\varphi}$ and the convolution operator $C_{\psi}$:
        \[A := M_{\varphi}C_{\psi}.\]
        Note that $A$ is an integral operator but not Hilbert--Schmidt in general. Then as in Example (a) one can easily see that both $M_{\varphi}$ and $C_{\psi}$ are in $\BDO^p(G)$ and in $\Fc^{-1} \BDO^p(\widehat{G}) \Fc$. This of course implies $A \in \BDO^p(G) \cap (\Fc^{-1} \BDO^p(\widehat{G}) \Fc)$. Let $(x_{\gamma},\nu_{\gamma})_{\gamma \in \Gamma}$ be a net in $\Xi$ that converges to some $z \in \partial\Xi$. Since
        \[\alpha_{(x,\nu)}(M_{\varphi}) = M_{\varphi(x^{-1} \cdot)} \quad \text{and} \quad \alpha_{(x,\nu)}(C_{\psi}) = C_{\psi\nu},\]
        we either get $\alpha_z(M_{\varphi}) = 0$ if $(x_{\gamma})_{\gamma \in \Gamma}$ escapes every compact set or $\alpha_z(C_{\psi}) = 0$ if $(\nu_{\gamma})_{\gamma \in \Gamma}$ escapes every compact set. Either way, we have $\alpha_z(A) = 0$ for all $z \in \partial \Xi$. Theorem \ref{thm:compactness} thus implies that $A$ is compact.
        \item[(c)] More generally, we can consider integral operators $A \from L^2(G) \to L^2(G)$ with suitably nice kernel $k_A \from G \times G \to \C$. Then $\widetilde{A} := \Fc A\Fc^{-1} \from L^2(\widehat{G}) \to L^2(\widehat{G})$ is also an integral operator with kernel
        \[k_{\widetilde{A}}(\nu,\mu) = \int_G\int_G k_A(x,y)\overline{\nu(x)}\mu(y) \, \mathrm{d}x \, \mathrm{d}y.\]
        Hence, for instance, if there are $g \in L^1(G)$ and $h \in L^1(\widehat{G})$ such that $\abs{k_A(x,y)} \leq g(y^{-1}x)$ and $\abs{k_{\widetilde{A}}(\nu,\mu)} \leq h(\mu^{-1}\nu)$, then $A \in \BDO^p(G) \cap (\Fc^{-1} \BDO^p(\widehat{G}) \Fc)$. Moreover,
        \[\alpha_{(x,\nu)}(A)f(y) = \int_G k_A(x^{-1}y,x^{-1}z)\nu(z^{-1}y)f(z) \, \mathrm{d}z.\]
        The compactness condition $\alpha_{(x,\nu)}(A) \to 0$ as $(x,\nu) \to \partial\Xi$ thus translates to a condition on the kernel of $A$. This very much resembles the matrix case discussed in Theorem \ref{thm:matrix_case}.
    \end{itemize}
\end{example}

\appendix

\section{Lca groups and Property A'}

The following is a metric-free version of \cite[Definition 2.3]{HaggerSeifert}.

\begin{definition}
    Let $G$ be a topological group. We say that $G$ has property $A'$ if for every $\epsilon > 0$ and every compact neighbourhood $H$ of $e \in G$ there is a collection $(\rho_j)_{j \in \Jc}$ of Borel measurable functions $\rho_j \from G \to [0,1]$ with the following properties:
    \begin{itemize}
        \item[(i)] It holds $\sum\limits_{j \in \Jc} \rho_j(g) = 1$ for all $g \in G$.
        \item[(ii)] There exists a compact set $K \subseteq G$ such that for every $j \in \Jc$ there is a $g \in G$ such that $\supp\rho_j \subseteq gK$.
        \item[(iii)] We have $\sum\limits_{j \in \Jc} \abs{\rho_j(g) - \rho_j(gh^{-1})} < \epsilon$ for all $g \in G$ and $h \in H$.
        \item[(iv)] For all compact subsets $K \subseteq G$ the set $\set{j \in \Jc : \supp \rho_j \cap K \neq \emptyset}$ is finite.
    \end{itemize}
\end{definition}

We note that the fourth property ensures that the possibly uncountable sums make sense.

\begin{theorem}
    Every lca group $G$ satisfies property $A'$ and the functions $\rho_j$ can be chosen in $S_0(G) \cap C_c(G)$.
\end{theorem}

Our proof is inspired by \cite{DeprezLi2015}, but since our locally compact groups are abelian, the situation is a bit easier and therefore we present a complete proof for convenience. Also note that for our purposes here, unlike in \cite{DeprezLi2015}, we do not need to assume that $G$ is second countable.

\begin{proof}
As every lca group is amenable by \cite[Proposition 0.15]{Paterson1988}, \cite[Theorem 4.10]{Paterson1988} implies that for every $\epsilon > 0$ and every compact subset $H \subseteq G$ there is a compact subset $K\subseteq G$ with $\lambda(K) > 0$ such that
\[\frac{\lambda(hK \Delta K)}{\lambda(K)} < \epsilon\]
for all $h \in H$. For every $g \in G$ define
\[\eta_g(h) := \frac{1}{\lambda(K)}\1_K(g^{-1}h) = \begin{cases} \frac{1}{\lambda(K)} & \text{if } g^{-1}h \in K,\\ 0 & \text{otherwise.}\end{cases}\]
Clearly, $\eta_g$ is measurable and
\[\int_G \eta_g(x) \, \mathrm{d}x = \int_{gK} \frac{1}{\lambda(K)} \, \mathrm{d}x = 1\]
for every $g \in G$. Moreover,
\[\int_G \abs{\eta_g(x) - \eta_{gh^{-1}}(x)} \, \mathrm{d}x = \int_{gK \Delta gh^{-1}K} \frac{1}{\lambda(K)} \, \mathrm{d}x = \frac{\lambda(gK \Delta gh^{-1}K)}{\lambda(K)} = \frac{\lambda(hK \Delta K)}{\lambda(K)} < \epsilon\]
for all $g \in G$ and $h \in H$.

Let $U \subseteq G$ be an open, relatively compact neighbourhood of $e$. Then $\set{gU : g \in G}$ is an open cover of $G$. By \cite[Chapter II, Theorem 8.13]{Hewitt_Ross}, every lca group is paracompact. Therefore $\set{gU : g \in G}$ has a locally finite, open refinement $\set{U_j : j \in \Jc}$ for some index set $\Jc$. That is, for every compact set $K' \subseteq G$ the set
\[\set{j \in \Jc : U_j \cap K' \neq \emptyset}\]
is finite. Let $\set{\alpha_j : j \in \Jc}$ be a partition of unity subordinate to $\set{U_j : j \in \Jc}$ and define
\[\rho_j(g) := \int_{U_j} \alpha_j(x)\eta_g(x) \, \mathrm{d}x = \frac{1}{\lambda(K)} (\alpha_j \ast \1_K)(g).\]
By \cite[Lemma 4.3(iii)]{jakobsen18}, we know that $\rho_j \in S_0(G)$. It is also clear that each $\rho_j$ has compact support. More precisely, if $g^{-1}h \notin K$ for all $h \in U_j$, then $\rho_j(g) = 0$ and it follows that the support of $\rho_j$ is contained in $U_jK^{-1} = g'UK^{-1}$ for some $g' \in G$. Moreover, each $\rho_j$ is continuous since
\[\abs{\rho_j(g) - \rho_j(h)} \leq \int_{U_j} \alpha_j(x)\abs{\eta_g(x) - \eta_h(x)} \, \mathrm{d}x \leq \int_G \abs{\eta_g(x) - \eta_h(x)} \, \mathrm{d}x < \epsilon\]
for $g,h \in G$ with $h^{-1}g \in H$. We further observe that if $\supp\rho_j \cap K' \neq \emptyset$ for some compact set $K' \subseteq G$, then $U_j \cap K'K \neq \emptyset$. But $K'K$ is compact and therefore by the local finiteness of $\set{U_j : j \in \Jc}$, the set $\set{j \in \Jc : \supp\rho_j \cap K' \neq \emptyset}$ is finite. This ensures that all the upcoming sums are actually finite. In particular,
\[\sum\limits_{j \in \Jc} \rho_j(g) = \int_G \eta_g(x) \, \mathrm{d}x = 1\]
for all $g \in G$. Finally,
\begin{align*}
    \sum\limits_{j \in \Jc} \abs{\rho_j(g) - \rho_j(gh^{-1})} &\leq \sum\limits_{j \in \Jc} \int_{U_j} \alpha_j(x)\abs{\eta_g(x) - \eta_{gh^{-1}}(x)} \, \mathrm{d}x\\
    &= \int_G \abs{\eta_g(x) - \eta_{gh^{-1}}(x)} \, \mathrm{d}x\\
    &< \epsilon
\end{align*}
for $h \in H$.
\end{proof}

\bibliographystyle{abbrv}
\bibliography{main}

\vspace{1cm}

\begin{multicols}{2}

\noindent
Robert Fulsche\\
\href{fulsche@math.uni-hannover.de}{\Letter ~fulsche@math.uni-hannover.de}
\\
\noindent
Institut f\"{u}r Analysis\\
Leibniz Universit\"at Hannover\\
Welfengarten 1\\
30167 Hannover\\
GERMANY\\

\noindent
Raffael Hagger\\
\href{hagger@math.uni-kiel.de}{\Letter ~hagger@math.uni-kiel.de}
\\
\noindent
Mathematisches Seminar\\
Christian-Albrechts-Universität zu Kiel\\
Heinrich-Hecht-Platz 6\\
24118 Kiel\\
GERMANY

\end{multicols}
\end{document}